\documentclass[reqno, 11pt]{amsart}
\usepackage{mathrsfs,graphicx,url, amssymb, pinlabel,bbm}
\usepackage{float, mathtools}
\usepackage[usenames,dvipsnames]{xcolor}
\usepackage[colorlinks=true,linkcolor=Black,citecolor=Black, urlcolor=Black]{hyperref}

\usepackage{pgfplots}
 \pgfplotsset{width=3cm, compat=1.16}

\usepackage[letterpaper,margin=1in]{geometry}

\newtheorem{thm}{Theorem}
\newtheorem{lem}[thm]{Lemma}
\newtheorem{as}{Assumption}
\theoremstyle{definition}

\numberwithin{thm}{section}
\numberwithin{equation}{section} 

\DeclareMathOperator \re {Re}
\DeclareMathOperator \im {Im}
\newcommand{\Real}{\mathbb{R}}
\newcommand{\Complex}{\mathbb{C}}
\newcommand{\ve}{\varepsilon}
\newcommand{\loc}{\operatorname{loc}}
\newcommand{\bbo}{\mathbbm{1}}
\newcommand{\Integers}{\mathbb{Z}}
\newcommand{\sgn}{\operatorname{sgn}}
\newcommand{\Rve}{R_{V_\ve}}
\newcommand{\Rvz}{R_{V_0}}
\newcommand{\proj}{\mathcal{P}}
\newcommand{\tl}{\tilde{\lambda}}
\newcommand{\tr}{\operatorname{tr}}
\newcommand{\mcC}{\mathcal{C}}

\title[Persistence and disappearance of negative eigenvalues in dimension two]{Persistence and disappearance of negative eigenvalues in dimension two}
\author{T. J. Christiansen, K. Datchev, and C. Griffin}
\address{Department of Mathematics, University of Missouri, Columbia, MO 65211 USA}
\email{christiansent@missouri.edu}
\address{Department of Mathematics, Purdue University, West Lafayette, IN 47907 USA}
\email{kdatchev@purdue.edu}
\address{Department of Mathematics, University of Pennsylvania, Philadelphia, PA 19104 USA}
\email{cmgrif@sas.upenn.edu}
\begin{document}
\begin{abstract}
We compute asymptotics of  eigenvalues approaching the bottom of the continuous spectrum, and associated resonances, for Schr\"odinger operators in dimension two. We distinguish persistent eigenvalues, which have associated resonances,  from disappearing ones, which do not.  We illustrate the significance of this distinction by computing corresponding scattering phase asymptotics and numerical Breit--Wigner peaks. We prove all of our results for circular wells, and extend some of them to more general problems using recent resolvent techniques.
\end{abstract}
\maketitle

\section{Introduction}

The transition between discrete and continuous spectrum at low energies poses a special challenge in two dimensional scattering. The setting of scattering by a finite well is important for the light it sheds here, with non-analytic behavior exhibited by the simplest example: if $\varepsilon$ is a coupling constant tending to zero, then the ground state eigenvalue 
goes to zero like $e^{-C/|\varepsilon|}$. This was shown in \cite{bb}, in \cite{simon}, and in  \cite[Section~45, Problem~2]{ll}.

The same behavior is observed for certain excited eigenvalues, while others are to leading order linear in the coupling constant. In the examples we study, the latter kind correspond to resonances near zero, and we call them \textit{persistent}, while the former kind do not and we call them \textit{disappearing}: see Section \ref{s:pereig}. There we also establish a connection between persistent eigenvalues and $p$-resonances and/or zero eigenvalues, and between disappearing eigenvalues and $s$-resonances.

We then turn to consequences for the scattering phase. In Section \ref{s:spintro} we compute low energy asymptotics, and in Section \ref{s:bwintro} we interpret our results numerically in terms of Breit--Wigner peaks. These illustrate in a striking way the significance of a persistent eigenvalue for the behavior of the scattering matrix at low energies.

Throughout we focus especially on the case of a circular well, and indicate how some results obtained in that setting by Bessel function analysis can be extended to more general problems using resolvent techniques from \cite{cdobs,cdgen,cdy}. 

\subsection{Persistence and disappearance of eigenvalues} \label{s:pereig}

Consider a Schr\"odinger operator $P=-\Delta +V $ on $\mathbb R^2$, where the potential function $V$ is bounded, compactly supported, and real valued. Denote by $V_\varepsilon$ a family of such potentials, depending smoothly on a real parameter $\varepsilon$. If $P_\varepsilon = -\Delta + V_\varepsilon$ has a bound state at $\varepsilon=0$, with energy $E_0<0$, then at nearby values of $\varepsilon$ there is a family of bound states with energies $E_\varepsilon<0$ depending smoothly on $\varepsilon$. This is the familiar problem of perturbation theory as discussed in Section 38 of \cite{ll}.

The problem becomes richer if the energy level approaches $0$, which is the beginning of the continuous spectrum. To analyze this case, we bring in the notion of resonance. 

In polar coordinates, a general solution to $(P-\lambda^2)u=0$ has a large $r$ expansion of the form
\begin{equation}\label{e:outdef}
u(r,\theta)=\sum_{\ell=-\infty}^\infty  \Big(c_\ell H_\ell^{(1)}(\lambda r) + d_\ell H_\ell^{(2)}(\lambda r)\Big) e^{i\ell \theta}.
\end{equation}
See Section \ref{s:bessel} below for a review of Bessel functions; \eqref{e:outdef} follows from \eqref{e:bessdef} and \eqref{e:bessdiff}.
If $\im \lambda > 0$, then $P$ has an eigenvalue at $E=\lambda^2$ if and only if there is a solution $u \not \equiv 0$ to  $(P-\lambda^2)u=0$ whose expansion \eqref{e:outdef} has all $d_\ell=0$; see \eqref{e:hankellarge}. 

To introduce the notion of resonances and outgoing  solutions of $Pu=0$, we need to work with a larger set on which $H_\ell^{(j)}(\lambda)$, $j=1,\;2$
is single-valued.   We cut $\Complex$ along the negative imaginary
axis and identify this with the set 
$\{\lambda\in \Complex \setminus \{0\}: -\pi/2 \leq \arg \lambda <3\pi/2\}.$ We use the natural branch of  $H_\ell^{(j)}(\lambda)$ which is continuous for $-\pi/2\leq \arg \lambda<3\pi/2$.
As will become apparent in Section \ref{s:lcr}, we are essentially making a  branch cut for the logarithm in a way that suits our initial focus on 
eigenvalues and resonances near the 
continuous spectrum, particularly in the fourth quadrant.

For $\lambda\in \Complex$ with $-\pi/2\leq \arg \lambda<3\pi/2$
we say that a solution $u$ to $(P-\lambda^2)u=0$ is \textit{outgoing} if its expansion \eqref{e:outdef} has all $d_\ell=0$. By continuity, we extend this condition to $\lambda=0$ by saying that a solution to $Pu=0$ is \textit{outgoing} if it has a large $r$ expansion of the form
\begin{equation}\label{e:outdef0}
u(r,\theta)=\sum_{\ell=-\infty}^\infty  c_\ell r^{-|\ell|} e^{i\ell \theta}. 
\end{equation}
We say that $\lambda \in \mathbb C \setminus -i(0,\infty)$ is a \textit{resonance} and $u\colon \mathbb R^2 \to \mathbb C$ is a \textit{resonant state} if $u$ is outgoing, $u \not \equiv 0$, and $(P-\lambda^2)u=0$. See Section \ref{s:resprelim} for further details. Note that $E=\lambda^2$ is an eigenvalue if and only if $\lambda$ is a resonance and there is a resonant state $u \in L^2$. This always occurs if $\im \lambda >0$, and it can occur for $\im \lambda \le 0$ only when $\lambda =0$, and then if and only if $c_{-1}=c_0=c_1=0$. 

More precisely, given an outgoing solution to $Pu=0$, one says that $P$ has an \emph{$s$-resonance}, and $u$ is an \emph{$s$-resonant state}, if $c_0 \ne 0$. One says that $P$ has a \emph{$p$-resonance}, and $u$ is a \emph{$p$-resonant state}, if $c_0=0$ but $|c_1|+|c_{-1}| > 0$.  One says that $P$ has a \emph{zero eigenvalue}, and $u$ is a \emph{zero eigenfunction}, if $c_0=c_{-1}=c_1=0$ but $u \not\equiv 0$.

Let $V_\varepsilon$ be a family of bounded, compactly supported potentials defined for $\varepsilon \in I$, where $I$ is an open interval containing $0$. Suppose there is a family of eigenvalues $E_\varepsilon$ of $P_\ve$, defined for $\varepsilon \in I$, $\varepsilon <0$, such that $E_\varepsilon \uparrow 0$ when $\varepsilon  \uparrow 0$. We say that $E_\varepsilon$ \textit{persists} along $I$ if there exists a continuous family of outgoing solutions $(P_\varepsilon - \lambda_\varepsilon^2)u_\varepsilon=0$, defined for all $\varepsilon \in I$, such that $\lambda_\varepsilon^2=E_\varepsilon$ for $\varepsilon<0$ and $-\pi/2\leq \arg \lambda<3\pi/2$.

\begin{thm}\label{t:persist}
Let $V_\varepsilon$ be a family of compactly supported real-valued radial functions in $L^\infty(\mathbb R^2)$,  defined for $\varepsilon \in I$, where $I$ is an open interval containing $0$. Suppose $\varepsilon \mapsto V_\varepsilon$ is smooth from $I$ to $L^\infty(\mathbb R^2)$, and that the derivative of this map at $\varepsilon=0$ is a  nonnegative function which is not identically zero. Suppose further that there is a family of energy levels $E_\varepsilon$ such that $E_\varepsilon \uparrow 0$ when $\varepsilon \uparrow 0$. 
\begin{enumerate}
\item If there exists a corresponding family of nonradial eigenfunctions, i.e. if $P_0$ has a $p$-resonance or zero eigenvalue, then $E_\varepsilon$ persists.
\item If every corresponding family of eigenfunctions is radial, i.e. if $P_0$ has an $s$-resonance but no $p$-resonance or zero eigenvalues, then $E_\varepsilon$ does not persist.
\end{enumerate}
\end{thm}

Our key example is the circular well, given by
\begin{equation}\label{e:circwelldef}
V(r)=-a^2 \bbo_{\rho}(r) = \begin{cases} - a^2, \qquad &r \le \rho, \\ 0, \qquad &r > \rho, \end{cases}
\end{equation}
for some positive $a$ and $\rho$; then we may consider $\rho>0$ fixed and $a = a_\varepsilon = j - \varepsilon$, where $j$ is any zero of a Bessel function.  Figure \ref{f:3foldres} illustrates persistence in the $\ell=1,\,2,\,3$ modes,  while Figure \ref{f:ell=0} illustrates disappearance in the $\ell=0$ mode.

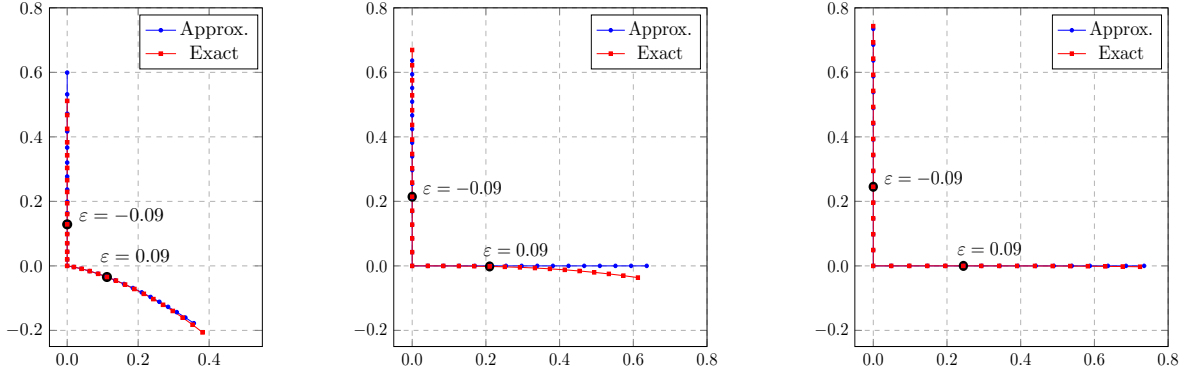
\begin{figure}[ht] 
  \centering
\resizebox{3.5 cm}{5 cm}{%
  \begin{tikzpicture}
\begin{axis}[
y tick label style={
        /pgf/number format/.cd,
            fixed,
            fixed zerofill,
            precision=1,
        /tikz/.cd
    },
    x tick label style={
        /pgf/number format/.cd,
            fixed,
            fixed zerofill,
            precision=1,
        /tikz/.cd
    },
          grid=major, 
          grid style={dashed,black!30}, 
tick label style={font=\small}, 
width=0.35\linewidth, 
height=0.5\linewidth,
xmin = -0.05, xmax = 0.55,
ymin = -0.25, ymax = 0.8,
]
\addplot+[mark options={scale=0.5}] table [x=ReGuess, y=ImGuess, col sep=comma] {dataset1.csv};
\addlegendentry{Approx.}
\addplot+[mark options={scale=0.5}] table [x=Re, y=Im, col sep=comma] {dataset1.csv};
\addlegendentry{Exact}
\node[label={[yshift=2.5ex,xshift=9ex]180:{$\varepsilon = 0.09$}},circle,fill,inner sep=2pt] at (axis cs:0.1119944,-0.0344571) {};
\node[label={[yshift=1ex,xshift=13ex]180:{$\varepsilon = -0.09$}},circle,fill,inner sep=2pt] at (axis cs:0.,0.1287651) {};
\end{axis}
\end{tikzpicture}}\hspace{1cm}
\resizebox{5 cm}{5 cm}{%
  \begin{tikzpicture}
\begin{axis}[
y tick label style={
        /pgf/number format/.cd,
            fixed,
            fixed zerofill,
            precision=1,
        /tikz/.cd
    },
    x tick label style={
        /pgf/number format/.cd,
            fixed,
            fixed zerofill,
            precision=1,
        /tikz/.cd
    },
          grid=major, 
          grid style={dashed,black!30}, 
tick label style={font=\small}, 
width=0.5\linewidth, 
height=0.5\linewidth,
xmin = -0.05, xmax = 0.8,
ymin = -0.25, ymax = 0.8,
]
\addplot+[mark options={scale=0.5}] table [x=ReGuess, y=ImGuess, col sep=comma] {dataset2.csv};
\addlegendentry{Approx.}
\addplot+[mark options={scale=0.5}] table [x=Re, y=Im, col sep=comma] {dataset2.csv};
\addlegendentry{Exact}
\node[label={[yshift=2ex,xshift=9ex]180:{$\varepsilon = 0.09$}},circle,fill,inner sep=2pt] at (axis cs:0.2100356,-0.0017315) {};
\node[label={[yshift=1ex,xshift=13ex]180:{$\varepsilon = -0.09$}},circle,fill,inner sep=2pt] at (axis cs:0.,0.2143996) {};
\end{axis}
\end{tikzpicture}}\hspace{1cm}
\resizebox{5 cm}{5 cm}{%
  \begin{tikzpicture}
\begin{axis}[
y tick label style={
        /pgf/number format/.cd,
            fixed,
            fixed zerofill,
            precision=1,
        /tikz/.cd
    },
    x tick label style={
        /pgf/number format/.cd,
            fixed,
            fixed zerofill,
            precision=1,
        /tikz/.cd
    },
          grid=major, 
          grid style={dashed,black!30}, 
tick label style={font=\small}, 
width=0.5\linewidth, 
height=0.5\linewidth,
xmin = -0.05, xmax = 0.8,
ymin = -0.25, ymax = 0.8,
]
\addplot+[mark options={scale=0.5}] table [x=ReGuess, y=ImGuess, col sep=comma] {dataset3.csv};
\addlegendentry{Approx.}
\addplot+[mark options={scale=0.5}] table [x=Re, y=Im, col sep=comma] {dataset3.csv};
\addlegendentry{Exact}
\node[label={[yshift=2ex,xshift=9ex]180:{$\varepsilon = 0.09$}},circle,fill,inner sep=2pt] at (axis cs:0.2445582,-0.0000141) {};
\node[label={[yshift=1ex,xshift=13ex]180:{$\varepsilon = -0.09$}},circle,fill,inner sep=2pt] at (axis cs:0.,0.2453089) {};
\end{axis}
\end{tikzpicture}}
\caption{Plot of $\lambda_\varepsilon$ in $\mathbb C \setminus -i(0,\infty)$ for $a^2 = j_{\ell-1,1}^2 - \varepsilon$, with $\ell=1$ (left) $\ell=2$ (middle) and $\ell=3$ (right). The values of $\varepsilon$ range from $\varepsilon = -0.81$ to $\varepsilon=0.81$. When $\varepsilon<0$, $\lambda_\varepsilon$ is on the positive imaginary axis and corresponds to an eigenvalue, when $\varepsilon>0$, $\lambda_\varepsilon$ is in the fourth quadrant and is a resonance.}\label{f:3foldres}
\end{figure}

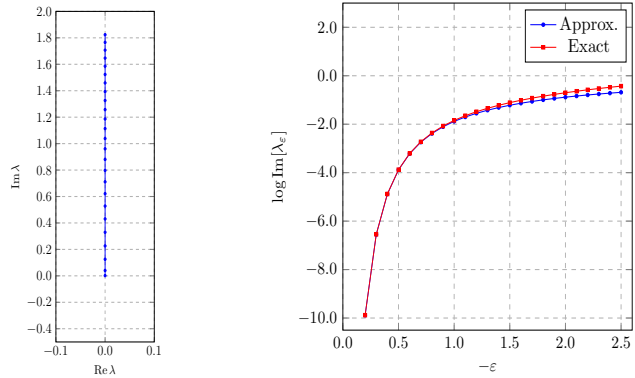
\begin{figure}[ht] 
  \centering
    \resizebox{2.1 cm}{5 cm}{%
  \begin{tikzpicture}
\begin{axis}[
y tick label style={
        /pgf/number format/.cd,
            fixed,
            fixed zerofill,
            precision=1,
        /tikz/.cd
    },
    x tick label style={
        /pgf/number format/.cd,
            fixed,
            fixed zerofill,
            precision=1,
        /tikz/.cd
    },
ylabel=$\operatorname{Im}\lambda$,
xlabel=$\operatorname{Re}\lambda$,
          grid=major, 
          grid style={dashed,black!30}, 
tick label style={font=\large}, 
width=0.3\linewidth, 
height=0.7\linewidth,
xmin = -.1, xmax =.1, xtick={-0.1,0,0.1},
ymin = -0.5, ymax = 2,
]
\addplot+[mark options={scale=0.5}] table [x=Re, y=Im, col sep=comma] {dataset0add.csv};
\end{axis}
\end{tikzpicture}}
\hspace{1cm}
 \resizebox{5 cm}{5 cm}{
  \begin{tikzpicture}
\begin{axis}[
y tick label style={
        /pgf/number format/.cd,
            fixed,
            fixed zerofill,
            precision=1,
        /tikz/.cd
    },
    x tick label style={
        /pgf/number format/.cd,
            fixed,
            fixed zerofill,
            precision=1,
        /tikz/.cd
    },
ylabel=\small{$\log\operatorname{Im}[\lambda_\varepsilon]$},
xlabel=\small{$-\varepsilon$},
          grid=major, 
          grid style={dashed,black!30}, 
tick label style={font=\small}, 
width=0.5\linewidth, 
height=0.5\linewidth,
xmin = 0, xmax =2.6,
ymin = -10.5, ymax = 3,
 xtick={0.0,0.5,1.0,1.5,2.0,2.5},
]
\addplot+[mark options={scale=0.5}] table [x=epsilon, y=ImGuess, col sep=comma] {dataset0.csv};
\addlegendentry{Approx.}
\addplot+[mark options={scale=0.5}] table [x=epsilon, y=Im, col sep=comma] {dataset0.csv};
\addlegendentry{Exact}
\end{axis}
\end{tikzpicture}}
\caption{Plot of $\lambda_\ve$ for $\ell=0$ mode, with $\ve \in \{-k\Delta\varepsilon : k=2,\ldots,25\}$, with $\Delta\varepsilon = 0.3$ in the first plot and $\Delta\varepsilon=0.1$ in the second plot. Approx. is $\lambda_\varepsilon = i \frac 2 \rho \exp(\frac 2 {\rho^2\varepsilon} - \gamma)$.}\label{f:ell=0}
\end{figure}


So far we have ignored multiplicities. Our next theorem takes them into account and also gives more detailed information about the behavior of the resonance near $\ve=0$. We work with joint eigenfunctions of $P_\varepsilon$ and $\partial_\theta$.

For real-valued $V$, the resonances of $-\Delta+V$ are symmetric in the sense that if $\lambda=|\lambda| e^{i\arg \lambda}$ is a resonance, so is $|\lambda|e^{i(\pi -\arg \lambda)}.$  In the next theorem we concentrate on resonances in the fourth quadrant for simplicity.

\begin{thm} \label{t:lt} Let $V_\varepsilon$ and $E_\ve$ satisfy the conditions of Theorem \ref{t:persist}.
Suppose for $\varepsilon<0$ $\psi_\varepsilon\in L^2(\Real)$, $\| \psi_\ve\|=1$, $(-\Delta+V_\varepsilon)\psi =E_\varepsilon \psi_\varepsilon$, $\partial _\theta \psi _\ve=i \ell \psi_\ve,$ and 
 $E_\varepsilon\uparrow 0$ as $\ve \uparrow 0$.
\begin{enumerate}
\item If $| \ell |>1$, then $P_0$ has a zero eigenvalue and  there is a constant $\alpha_0(\ell)>0$ and an $\varepsilon _0>0$ so that for $0<\varepsilon<\varepsilon_0$, $P_\varepsilon$ has a resonance in the fourth quadrant satisfying 
$$\lambda_\varepsilon = \sqrt{ \alpha_0(\ell) \varepsilon}+O(\ve^{3/2}(1+\delta_{|\ell|,2}\log \ve)).$$  Counting with multiplicity, there is a resonance for each $\ell$ satisfying these conditions.
\item If $|\ell | =1$, then $P_0$ has a $p$-resonance and there are constants $\alpha_0>0$, $b=\re b +i\pi/2$ and an $\varepsilon _0>0$ so that for $0<\varepsilon<\varepsilon_0$, $P_\varepsilon$ has a resonance
 in the fourth quadrant satisfying $$\lambda_{\ve} = \exp\left( b+ \frac{1}{2} W_{-1}(2\ve \alpha_0 e^{-2 \re b})\right)(1+O(\ve)).$$   Counting with multiplicity, there is a resonance for each $\ell$ satisfying these conditions.
\end{enumerate}
\end{thm}
In fact, as we shall describe in Section \ref{s:lcr}, in the setting of Theorem \ref{t:lt} for each such $\ell$ there are many resonances which are related to the negative eigenvalue, and they exist for both 
negative and positive $\varepsilon$. Moreover, we find explicit expressions for $\alpha_0$ in terms of the null space of $-\Delta +V_0$ and $(\partial_\ve V_\ve)\upharpoonright_{\ve=0}$ in the general case,
see Lemmas \ref{l:evpert} and \ref{l:pwrg}  in Section \ref{s:lcr}.   In case 
 $V_\ve=(-a_0^2+\ve)\bbo_\rho$  these $\alpha_0$ can be written
 in terms of $\rho$ and $\ell,$ see 
 Lemmas \ref{l:2c} and \ref{l:pwrc} in Section \ref{s:cwr}.
 In addition, Lemmas \ref{l:evpert} and \ref{l:pwrg}
 include a further term in
the expansion for $\ell\not = \pm 2$ that helps explain the 
difference between the center and right graphs in Figure \ref{f:3foldres}. 
 

\subsubsection{Further background and context}

The paper \cite{ha} surveys some results in perturbation of
eigenvalues and resonances.  We mention just a few papers most
closely related to studying eigenvalues and resonances
of Schr\"{o}dinger operators near
$0$, particularly in dimension two, and direct
the reader to \cite{ha} for further results and references.

Klaus and Simon \cite{klaussimon} and Holden \cite{ho}   study the possible expansions
of $E_\ve\uparrow 0$ as $\ve\uparrow 0$ in any dimension
and for 
potentials which need not be radially symmetric.  The results of
\cite{klaussimon} show that dimension $2$ holds the most different possibilities.  These papers do not explore whether or not the eigenvalues persist.  Compared
to \cite{klaussimon}, our results hold for a smaller class of potentials, but our leading terms are more detailed and cover resonances as well.  The paper \cite{mdwm} contains
some nice diagrams of  numerically computed eigenvalues and
resonances for the Schr\"{o}dinger operator
with circular well potential, $-\Delta -a^2 \bbo_{\rho}$, on $\Real^2$.  These eigenvalues and resonances 
are not just those very near $0$, and the diagrams indicate how the resonances and eigenvalues move as $a$ varies. In \cite{gk} Grigis and Klopp study 
eigenvalues and resonances 
 near the boundary of the continuous spectrum
 for perturbations of 
 of Schr\"{o}dinger operators with periodic 
potentials.  Their results do not include
dimension two, but we note that their results
in even dimension at least four resemble our Lemma \ref{l:evpert}, on perturbing an eigenvalue at $0$. 

For Schr\"{o}dinger operators on $\Real^3$ depending 
on a parameter the papers \cite{ra,gh} study the transition from eigenvalues
to resonances at the bottom of the continuous spectrum.
In \cite{jn}, Jensen and Nenciu study the perturbation of 
an eigenvalue at the bottom of the continuous spectrum using a
time-dependent notion of resonances.  Their
results are quite general, and apply to Schr\"{o}dinger operators in dimension $3$, and to certain two channel operators 
in dimension $1$ and $3$ with potentials which decay at infinity but need not be compactly supported.  They
do not study the even-dimensional case.  In fact, near the end of \cite{ha}, Harrell asks ``Can the restriction to odd dimensions in articles such as \cite{jn} be relaxed?"  Our results
in Theorem \ref{t:lt} and Section \ref{s:riemsurf}
are a step towards doing so.

Most recently, perturbations of resonances at the bottom of the continuous spectrum have seen  applications in the study of subwavelength resonator systems: see \cite{adh} for a review and references. In that setting, one seeks to use small objects to strongly scatter waves with comparatively large wavelengths. In \cite{adh}, and in references therein, asymptotics are computed for resonances obtained by perturbing off of a resonance at zero for a certain non-self-adjoint problem.

Our results in Section \ref{s:lcr} use results about the 
resolvent expansion at $0$ for a Schr\"{o}dinger operator 
on $\Real^2.$  Results in this direction include \cite{vai89, bgd, jn01, stwa, cdgen}.

\subsection{Scattering phase asymptotics}\label{s:spintro}
Theorems \ref{t:persist} and \ref{t:lt} examine the importance of the different kinds of resonance or eigenvalue at zero energy for perturbation theory. If we consider a fixed circular well, we can illustrate their importance for scattering phase asymptotics.  The scattering phase is defined by
\begin{equation}\label{e:sigmadef}
\sigma(\lambda) = \frac 1 {2\pi i} \log \det S(\lambda) = \sum_{\ell} \sigma_\ell(\lambda),
\end{equation}
where $S(\lambda)$ is the scattering matrix and the $\sigma_\ell(\lambda)$ are the phase shifts, i.e. $e^{2\pi i \sigma_\ell(\lambda)}$ are the eigenvalues of  $S(\lambda)$, normalized so that $\sigma_\ell(0) = 0$. When $V$ is radial, the set $\{e^{i\ell\theta}\}_{\ell \in \mathbb Z}$ is an orthogonal basis for $S(\lambda)$, and we may take each $\sigma_\ell(\lambda)$ to be the phase shift corresponding to $e^{i\ell\theta}$. See Sections~2.6 and 3.9 of \cite{dyatlovzworski} for a more general introduction.

\begin{thm} \label{t:spintro}Let $P = -\Delta - a^2 \bbo_\rho$ for some positive $a$ and $\rho$. Let $\gamma=-\Gamma'(1) = 0.577\dots$ Then, as $\lambda \downarrow 0$
\begin{enumerate}
\item If $P$ has no resonance at zero, then there is a real $b$ such that
\[
\sigma'(\lambda) = \frac {-1}\lambda \frac 2{4(\log \lambda + b)^2 +  {\pi^2}} + O(\lambda),
\]
\item If $P$ has a $p$-resonance at zero, then 
\[
\sigma'(\lambda) = \frac {-1}\lambda \bigg(\frac 2{4(\log \lambda +\log(\rho/2)+\gamma)^2 +  {\pi^2}} + \frac 1{(\log \lambda + \log(\rho/2)+\gamma - \frac 12)^2 + \frac {\pi^2}4}\bigg) + O(\lambda),
\]
\item If $P$ has an $s$-resonance at zero, then
\[
\sigma'(\lambda) = - \frac 32 \rho^2 \lambda + O(\lambda^3 \log \lambda).
\]
\end{enumerate}
\end{thm}
Theorem \ref{t:mssp} gives an expression for the constant $b$,
and Figure \ref{fig:sigPA} provides an illustration.

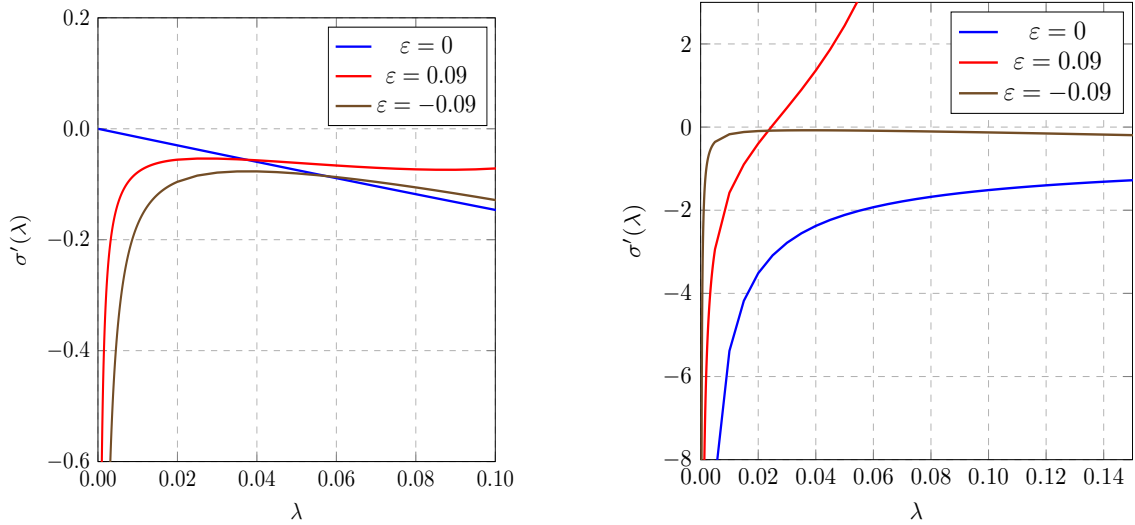
\begin{figure}[ht]
  \centering
  \resizebox{7 cm}{7 cm}{
  \begin{tikzpicture}
\begin{axis}[
y tick label style={
        /pgf/number format/.cd,
            fixed,
            fixed zerofill,
            precision=1,
        /tikz/.cd
    },
    x tick label style={
        /pgf/number format/.cd,
            fixed,
            fixed zerofill,
            precision=2,
        /tikz/.cd
    },
ylabel=\small{$\sigma'(\lambda)$},
xlabel=\small{$\lambda$},
          grid=major, 
          grid style={dashed,black!30}, 
tick label style={font=\small}, 
width=0.5\linewidth, 
height=0.5\linewidth,
xmin = -0, xmax = 0.1,
ymin = -0.6, ymax = 0.2,
]
\addplot+[line width = 1pt, mark = none, mark options={scale=1}] table [x=lambda, y=res, col sep=comma] {l0table_sig.csv};
\addlegendentry{$\ve=0$}
\addplot+[line width = 1pt, mark = none, mark options={scale=0.5}] table [x=lambda, y=above, col sep=comma] {l0table_sig.csv};
\addlegendentry{$\ve=0.09$}
\addplot+[line width = 1pt, mark = none, mark options={scale=0.5}] table [x=lambda, y=below, col sep=comma] {l0table_sig.csv};
\addlegendentry{$\ve=-0.09$}
\end{axis}
\end{tikzpicture}}\hspace{1cm}
  \resizebox{7 cm}{7 cm}{ 
  \begin{tikzpicture}
\begin{axis}[
y tick label style={
        /pgf/number format/.cd,
            fixed,
            fixed zerofill,
            precision=0,
        /tikz/.cd
    },
    x tick label style={
        /pgf/number format/.cd,
            fixed,
            fixed zerofill,
            precision=2,
        /tikz/.cd
    },
ylabel=\small{$\sigma'(\lambda)$},
xlabel=\small{$\lambda$},
          grid=major, 
          grid style={dashed,black!30}, 
tick label style={font=\small}, 
width=0.5\linewidth, 
height=0.5\linewidth,
xmin = -0, xmax = 0.15,
ymin = -8, ymax = 3,
]
\addplot+[line width = 1pt, mark = none, mark options={scale=1}] table [x=lambda, y=res, col sep=comma] {l1table_sig.csv};
\addlegendentry{$\ve=0$}
\addplot+[line width = 1pt, mark = none, mark options={scale=0.5}] table [x=lambda, y=above, col sep=comma] {l1table_sig.csv};
\addlegendentry{$\ve=0.09$}
\addplot+[line width = 1pt, mark = none, mark options={scale=0.5}] table [x=lambda, y=below, col sep=comma] {l1table_sig.csv};
\addlegendentry{$\ve=-0.09$}
\end{axis}
\end{tikzpicture}}
\caption{Graphs of $\sigma'(\lambda)$ for $-\Delta - a^2 \bbo_1$, with $a^2 = j_{1,1}^2-\ve$ on the left (perturbing a zero eigenvalue) and $a^2 = j_{0,1}^2-\ve$ (perturbing a $p$-resonance at zero) on the right. The cases $\varepsilon \ne 0$ all correspond to case (1) of Theorem \ref{t:spintro}. Case (2) of Theorem \ref{t:spintro} is $\varepsilon=0$ on the right and Case (3) of Theorem \ref{t:spintro} is $\varepsilon=0$ on the left. The behavior for larger $\lambda$ is in Figure \ref{fig:sigP0}.}\label{fig:sigPA}
\end{figure}

\subsubsection{Further background and context}
The behavior near $0$ of the scattering phase is largely
determined by the presence or absence of different kinds of
resonance or eigenvalue at $0$, and this, in turn,
is related to the low-energy behavior of the resolvent.
This is true not only for Schr\"{o}dinger operators with compactly supported potentials (without requiring the potential to be radial), but for general 
self-adjoint black-box compactly supported perturbations  of
$-\Delta$ on $\Real^2$.
In the absence of $0$-energy eigenvalues or resonances, 
the proof of \cite[Theorem 3]{cdobs} combined with \cite[Theorem~2]{cdgen}, shows that 
 there is a real number $b$ such that 
\begin{equation}\label{eq:gspe}
\sigma'(\lambda)=\frac{-1}{\lambda}\frac{2}{4(\log \lambda +b)^2+\pi^2}+O(\lambda).
\end{equation}
Here $\sigma$ is the scattering phase for the black-box operator $P$.

For the Dirichlet Laplacian on an external domain in $\Real^2$, the universal behavior $\sigma(\lambda)= \frac{1}{2}(\log \lambda)^{-1} + O((\log \lambda)^{-2})$ was proved first by Hassell and Zelditch in \cite{haze}.  Their results were further refined in   \cite{mcg, stwa,cdobs}. The recent work \cite{gmwz} provides a wealth of numerical examples.

A Schr\"odinger operator on $\Real^2$ generically has neither eigenvalue nor resonance at $0$, in which case again
$\sigma(\lambda)= \frac 12 (\log \lambda)^{-1}+ O((\log \lambda)^{-2})$.
This, or at least very closely related results, has been observed
in \cite{bg84,bgd}; see \cite{av, kmrw} for further related results and more references.  Moreover \cite{bg84, kmrw} 
relate the next term in the expansion to the scattering length.

The constant $b$ appearing in \eqref{eq:gspe} can be 
identified using the asymptotic behavior of a function in the 
null space of $P$. Under the
assumption that $P$ has no $s$-resonance, 
there is a unique function  $G$ which is
locally in the domain of $P$ and which satisfies
$PG=0$ and $G(x)-\log |x| =O(1) $ as $|x|\rightarrow \infty$. Then 
$$b=-\log 2 +\gamma +C(P),\; \text{where}\; C(P)=\lim_{x\rightarrow \infty} (\log |x|-G(x)).$$
For the case of scattering by a nontrivial smooth obstacle with Dirichlet boundary conditions,  $C(P)$ is the logarithm of the logarithmic capacity of the obstacle-- see \cite{cdobs} for further discussion. For the case of a radial potential $V$
with $P_V=-\Delta+V$, $C(P_V)$ is the negative reciprocal of the scattering length as defined in \cite{bg84}.  See \cite[Sections II and VII]{bg84} for a discussion of other definitions (both equivalent and not) of the scattering length in two dimensions.

The paper \cite{bgd} obtains, for much more general potentials $V$, expressions for the leading term of the scattering matrix which may be used to derive the leading term of the
scattering phase (with $\sigma(0)=0$), but
with worse errors than those coming
from Theorem \ref{t:spintro}.  The results of \cite{bgd}
allow the possibility of an eigenvalue or resonance at $0$.



\subsection{Breit--Wigner peaks}\label{s:bwintro}
In this section we use Breit--Wigner formulas to illustrate the significance of a persistent eigenvalue for scattering.

We begin with a one-dimensional example to set the stage. Let $\sigma(\lambda)$ be the scattering phase of $- \frac{d^2}{dx^2} + V$, with $V$ compactly supported. Let $L$ be the length of the convex hull of the support of $V$, and let $\text{Res}^*$ be the set of nonzero resonances of $-\frac {d^2}{dx^2} + V$. The Breit--Wigner formula  says that
\begin{equation}\label{e:bw1d}
\sigma'(\lambda) = \frac {-L} \pi + \sum_{\lambda_k \in \text{Res}^*} \frac {-\im \lambda_k}{\pi |\lambda - \lambda_k|^2}.
\end{equation}
This formula, and others like it, are deduced from 
factorizations of $\det S(\lambda)$: see \cite[Theorem~2.20]{dyatlovzworski} for the full-line problem, and  \cite[Theorem 1.1]{korotyaev} for the half-line problem. 


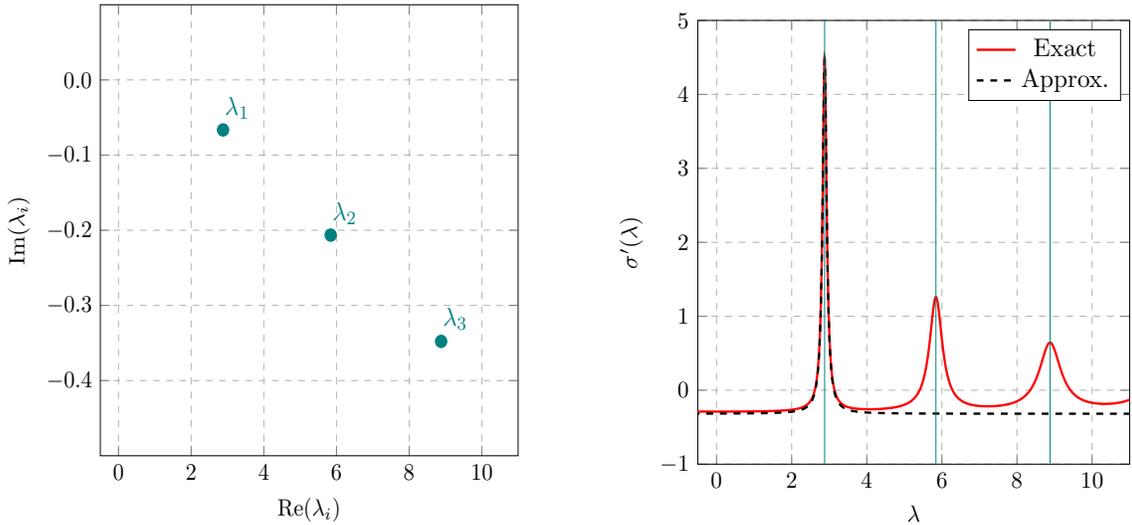
\begin{figure}[ht]
  \centering
  \resizebox{7 cm}{7 cm}{
  \begin{tikzpicture}
\begin{axis}[
y tick label style={
        /pgf/number format/.cd,
            fixed,
            fixed zerofill,
            precision=1,
        /tikz/.cd
    },
    x tick label style={
        /pgf/number format/.cd,
            fixed,
            fixed zerofill,
            precision=0,
        /tikz/.cd
    },
ylabel=\small{$\im(\lambda_i)$},
xlabel=\small{$\re(\lambda_i)$},
          grid=major, 
          grid style={dashed,black!30}, 
tick label style={font=\small}, 
width=0.5\linewidth, 
height=0.5\linewidth,
xmin = -0.5, xmax = 11,
ymin = -0.5, ymax = 0.1,
xtick={0,2,4,6,8,10},
ytick={-0.4,-0.3,-0.2,-0.1,0},
]
\node[color = teal, label={[color = teal, yshift=2ex,xshift=4ex]180:{$\lambda_1$}},circle,fill,inner sep=2pt] at (axis cs:2.877577458,-0.0665106) {};
\node[color = teal, label={[color = teal, yshift=2ex,xshift=4ex]180:{$\lambda_2$}},circle,fill,inner sep=2pt] at (axis cs:5.841379586,-0.20648) {};
\node[color = teal, label={[color = teal, yshift=2ex,xshift=4ex]180:{$\lambda_3$}},circle,fill,inner sep=2pt] at (axis cs:8.880653554,-0.34784182) {};
\end{axis}
\end{tikzpicture}}\hspace{1cm}
  \resizebox{7 cm}{7 cm}{ 
  \begin{tikzpicture}
\begin{axis}[
y tick label style={
        /pgf/number format/.cd,
            fixed,
            fixed zerofill,
            precision=0,
        /tikz/.cd
    },
    x tick label style={
        /pgf/number format/.cd,
            fixed,
            fixed zerofill,
            precision=0,
        /tikz/.cd
    },
ylabel=\small{$\sigma'(\lambda)$},
xlabel=\small{$\lambda$},
          grid=major, 
          grid style={dashed,black!30}, 
tick label style={font=\small}, 
width=0.5\linewidth, 
height=0.5\linewidth,
xmin = -0.5, xmax = 11,
ymin = -1, ymax = 5,
xtick={0,2,4,6,8,10},
ytick={-1,0,1,2,3,4,5},
]
\addplot+[color = red, line width = 1pt, mark = none, mark options={scale=1}] table [x=lambda, y=sigma, col sep=comma] {tab4.csv};
\addlegendentry{Exact}
\addplot+[dashed, color = black, line width = 1pt, mark = none, mark options={scale=0.5}] table [x=lambda, y=approx, col sep=comma] {tab4.csv};
\addlegendentry{Approx.}
\draw[color=teal] (2.877577458,-0.99) -- (2.877577458,4.99);
\draw[color=teal] (5.841379586,-0.99) -- (5.841379586,4.99);
\draw[color=teal] (8.880653554,-0.99) -- (8.880653554,4.99);
\end{axis}
\end{tikzpicture}}
\caption{On the left, the first three resonances of $-\frac {d^2}{dx^2} + 10\delta_1$ on the half line. On the right, a graph of $\sigma'(\lambda)$. The vertical lines are at $\re \lambda_1$, $\re \lambda_2$, $\re \lambda_3$, and the dashed curve is the approximation $\frac{-1}\pi - \frac{\im \lambda_1}{\pi|\lambda - \lambda_1|^2}$. 
} \label{f:delta}
\end{figure}

A simple and striking example is $P=-\frac {d^2}{dx^2} + a\delta_1$ on $\mathbb R_+$, with Dirichlet boundary condition at~$0$. The resonances (obtained by solving $(P-\lambda^2)u=0$ subject to $u(x) =e^{i\lambda x}$ for $x>1$) are given in terms of the Lambert \textit{W} function (see \eqref{e:wn}) by $\lambda_k = \frac 1 {2i} (a-W_{-k}(ae^a))$, where $k$ varies over the nonzero integers. The derivative of the scattering phase (obtained by solving $(P-\lambda^2)u=0$ subject to $u(x) = e^{-i\lambda x} + e^{2\pi i \sigma(\lambda)}e^{i\lambda x}$ for $x>1$) is given by $\sigma'(\lambda) = -\frac 1 \pi + \frac 1 \pi \frac{a + \lambda^2 \csc^2\lambda}{(a+\lambda \cot \lambda)^2+\lambda^2}$. Figure \ref{f:delta} shows the first three Breit--Wigner peaks, corresponding to the first three resonances. Moreover, at low energies, up to and including the first peak, we see that   $\sigma'(\lambda) \approx \frac{-L}\pi - \frac{\im \lambda_1}{\pi|\lambda - \lambda_1|^2}$.

In higher dimensions, Breit--Wigner formulas are more complicated. In $\mathbb R^d$, $d \ge 2$, work of Petkov and Zworski \cite{petkovzworski} shows that $\sigma'(\lambda)=\sum_{|\lambda_k-\lambda|<1} \frac {-\im \lambda_k}{\pi |\lambda - \lambda_k|^2} + O(\lambda^{d-1})$ as $\lambda \to +\infty$.  If $d$ is odd, an analog of \eqref{e:bw1d}, with $-L/\pi$ replaced by a polynomial of degree $\le d$ and with the sum $\sum \frac {-\im \lambda_k}{\pi |\lambda - \lambda_k|^2}$ modified so as to converge, follows from the factorization of $\det S(\lambda)$ in \cite[Theorem 3.54]{dyatlovzworski}.

In even dimensions, the non-resonance contribution (i.e the replacement of the term $-L/\pi$ in \eqref{e:bw1d}) is not known; by  Theorem \ref{t:spintro}  it is sometimes singular as $\lambda \to 0$. Moreover, it is not clear how to modify the sum $\sum \frac {-\im \lambda_k}{\pi |\lambda - \lambda_k|^2}$ to account for resonances on the Riemann surface $\Lambda$ (see Section \ref{s:riemsurf}). Nevertheless, in Figure~\ref{fig:sigP0}, which is typical for the circular well family, we see that a resonance coming from a persistent eigenvalue produces a large Breit--Wigner peak, easily seen despite the $\lambda \to 0$ singularity. See also \cite[Figure 2]{gmwz} for  examples of Breit--Wigner peaks for  obstacle scattering in $\mathbb R^2$, with significant   peaks corresponding to resonances generated by significant trapping of the  billiard flow. 

\begin{figure}[ht]
  \centering
  \resizebox{7 cm}{7 cm}{
  \begin{tikzpicture}
\begin{axis}[
y tick label style={
        /pgf/number format/.cd,
            fixed,
            fixed zerofill,
            precision=1,
        /tikz/.cd
    },
    x tick label style={
        /pgf/number format/.cd,
            fixed,
            fixed zerofill,
            precision=2,
        /tikz/.cd
    },
ylabel=\small{$\sigma'(\lambda)$},
xlabel=\small{$\lambda$},
          grid=major, 
          grid style={dashed,black!30}, 
tick label style={font=\small}, 
width=0.5\linewidth, 
height=0.5\linewidth,
xmin = -0, xmax = 1,
ymin = -1, ymax = 1,
]
\addplot+[line width = 1pt, mark = none, mark options={scale=1}] table [x=lambda, y=res, col sep=comma] {l0table_sig.csv};
\addlegendentry{$\ve=0$}
\addplot+[line width = 1pt, mark = none, mark options={scale=0.5}] table [x=lambda, y=above, col sep=comma] {l0table_sig.csv};
\addlegendentry{$\ve=0.09$}
\addplot+[line width = 1pt, mark = none, mark options={scale=0.5}] table [x=lambda, y=below, col sep=comma] {l0table_sig.csv};
\addlegendentry{$\ve=-0.09$}
 \draw[color=teal]  (0.2100356,-1) -- (0.2100356,1);
  \draw[domain=0.23:1, smooth, dashed] plot[id=exp] function{0.0017315/((x-0.2100356)**2 +  0.0017315**2)  - 0.8 * sqrt(x)};
    \draw[domain=0:0.18, smooth, dashed] plot[id=exp] function{0.0017315/((x-0.2100356)**2 +  0.0017315**2)  - 0.8 * sqrt(x)};
\end{axis}
\end{tikzpicture}}\hspace{1cm}
  \resizebox{7 cm}{7 cm}{ 
  \begin{tikzpicture}
\begin{axis}[
y tick label style={
        /pgf/number format/.cd,
            fixed,
            fixed zerofill,
            precision=0,
        /tikz/.cd
    },
    x tick label style={
        /pgf/number format/.cd,
            fixed,
            fixed zerofill,
            precision=2,
        /tikz/.cd
    },
ylabel=\small{$\sigma'(\lambda)$},
xlabel=\small{$\lambda$},
          grid=major, 
          grid style={dashed,black!30}, 
tick label style={font=\small}, 
width=0.5\linewidth, 
height=0.5\linewidth,
xmin = -0, xmax = 1,
ymin = -4, ymax = 4,
]
\addplot+[line width = 1pt, mark = none, mark options={scale=1}] table [x=lambda, y=res, col sep=comma] {l1table_sig.csv};
\addlegendentry{$\ve=0$}
\addplot+[line width = 1pt, mark = none, mark options={scale=0.5}] table [x=lambda, y=above, col sep=comma] {l1table_sig.csv};
\addlegendentry{$\ve=0.09$}
\addplot+[line width = 1pt, mark = none, mark options={scale=0.5}] table [x=lambda, y=below, col sep=comma] {l1table_sig.csv};
\addlegendentry{$\ve=-0.09$}
\draw[color=teal] (0.1119944,-4) -- (0.1119944,4);
  \draw[domain=0.12:1, smooth, dashed] plot[id=exp] function{0.0344571/((x-0.1119944)**2 +  0.0344571**2) + log(x)};
    \draw[domain=0:0.1, smooth, dashed] plot[id=exp] function{0.0344571/((x-0.1119944)**2 +  0.0344571**2) +log(x)};
\end{axis}
\end{tikzpicture}}
\caption{Graphs of $\sigma'(\lambda)$ for $-\Delta - a^2 \bbo_1$, with $a^2 = j_{1,1}^2-\ve$ on the left (perturbing a zero eigenvalue) and $a^2 = j_{0,1}^2-\ve$ (perturbing a $p$-resonance at zero) on the right.  The vertical lines are at $\re \lambda_\ve$, with $\lambda_\ve$ as  in Figure \ref{f:3foldres}. Each $\ve=0.09$ curve has a maximum when $\lambda \approx \re \lambda_\varepsilon$, with $\sigma'(\re \lambda_\varepsilon)\approx 367.37$ on the left, and  $\sigma'(\re \lambda_\varepsilon)\approx 17.048$ on the right. The dashed curves are Breit--Wigner approximations $ \frac{-\im \lambda_\ve}{\pi|\lambda - \lambda_\ve|^2} - 0.8\sqrt \lambda $  and  $\frac{-\im \lambda_\ve}{\pi|\lambda - \lambda_\ve|^2} + \log \lambda$, with the non-resonance terms $\sqrt \lambda$ and $\log \lambda$ (meant to be analogous to $-L/\pi$ in \eqref{e:bw1d}) chosen ad hoc by eye to improve the fit.}\label{fig:sigP0}
\end{figure}
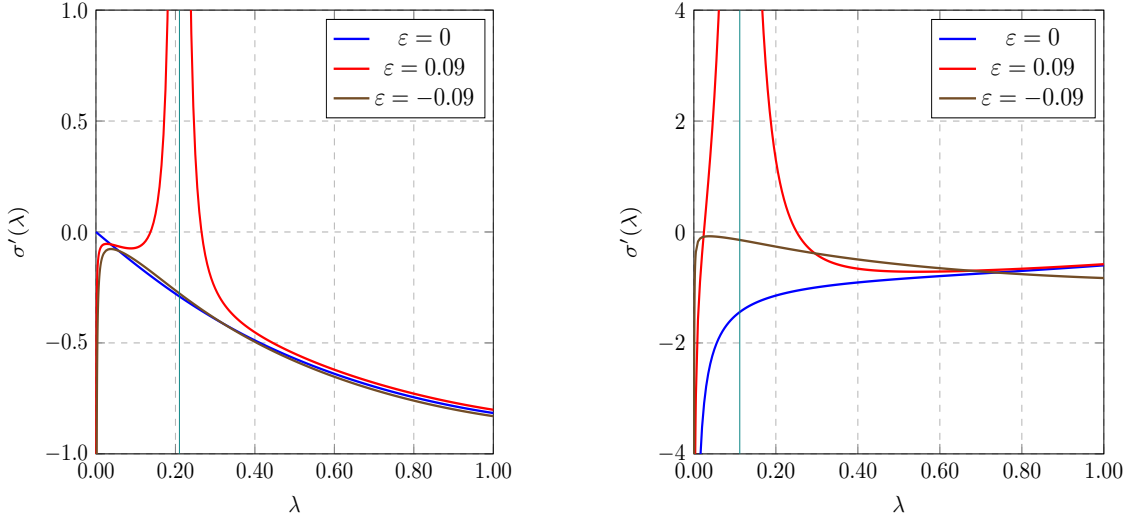

\section{Preliminaries}

\subsection{Bessel and Hankel functions}\label{s:bessel} We use the standard notation of \cite[10.2(ii) and 10.4.3.]{dlmf}:
\begin{equation}\label{e:bessdef}\begin{split}
J_\nu(z)=(z/2)^\nu\sum_{k=0}^\infty\frac{(-z^2/4)^k}{k!\Gamma(\nu+k+1)}, &\quad Y_\ell(z) = \frac 1 \pi \partial_\nu J_\nu(z)|_{\nu = \ell} + \frac {(-1)^n}\pi \partial_\nu J_{-\nu}(z)|_{\nu=-\ell}, \\
H_\ell^{(1)}(z) = J_\ell(z) + i Y_\ell(z), &\quad  
H_\ell^{(2)}(z) = J_\ell(z)  - i Y_\ell(z).
\end{split}\end{equation}
If $\mathcal{C}_\ell$ is one of $J_\ell$, $Y_{\ell}$, $H_{\ell}^{(1)},$  or $H_{\ell}^{(2)}$, then $\mathcal{C}_\ell$ solves the differential equation \cite[10.2(i)]{dlmf}
\begin{equation}\label{e:bessdiff}
z^2\mathcal{C}_{\ell}''(z) + z \mathcal{C}_{\ell}'(z) + (z^2 - \ell^2)\mathcal{C}_{\ell}(z)=0,
\end{equation}
and we also have the recurrence relations   \cite[10.6(i)]{dlmf}
\begin{equation}\label{eq:recur} \begin{split}
    \mcC_{\ell -1}(z) +\mcC_{\ell+1}(z)=\frac{2\ell}{z}\mcC_{\ell}(z),  &\qquad \mcC_{\ell-1}(z)-\mcC_{\ell+1}(z)=2\mcC_\ell'(z), \\
\mcC_\ell'(z)= \mcC_{\ell -1}(z)-\frac{\ell}{z}\mcC_{\ell}(z),  &\qquad \mcC_{\ell}'(z)=-\mcC_{\ell+1}(z)+\frac{\ell}{z}\mcC_{\ell}(z).
\end{split} \end{equation}
As $|z| \to \infty$, $z \in \mathbb C \setminus -i(0,\infty)$, we have  \cite[10.2.5 and 10.2.6]{dlmf}
\begin{equation}\label{e:hankellarge}
H^{(1)}_\ell(z) \sim \sqrt{\tfrac 2 {\pi z}} e^{i(z - \frac{\ell \pi}2 - \frac \pi 4)}, \qquad H^{(2)}_\ell(z) \sim \sqrt{\tfrac 2 {\pi z}} e^{-i(z - \frac{\ell \pi}2 - \frac \pi 4)}.
\end{equation}
As $|z| \to 0$, we have \cite[10.8]{dlmf}
\begin{equation}\label{eq:Yexp}
    \begin{aligned}
        Y_0(z) &= \frac{2}{\pi}[\log(z/2) + \gamma] + O(z^2\log z), \quad \gamma = -\Gamma'(1) \approx 0.5772,\\
        Y_1(z) &= -\frac {2}{\pi z} + \frac {z \log z}{\pi} + O(z)\\
        Y_\ell(z) &= -\frac{(\ell-1)! 2^\ell}{\pi z^\ell} - \frac{(\ell-2)!2^{\ell-2}}{\pi z^{\ell-2}} + O(z^{4-\ell}(1+\delta_{2,\ell} \log z)) , \qquad \ell \ge 2.
    \end{aligned}
\end{equation} 

\subsection{Lambert \textit{W} function}\label{s:lambert} If $x \in \mathbb C\setminus\{0\}$, then $we^w=x$ has infinitely many solutions. Following \cite{wfcn}, we denote these by $W_n(x)$, with $n$ varying over the integers. If $n \ne 0$, then
\begin{equation}\label{e:wn}
W_n(x) = L_n(x) - \log( L_n(x) ) + \sum_{k=0}^\infty \sum_{m=1}^\infty c_{km} \frac {(\log(L_n(x)))^m}{L_n(x)^{k+m}},
\end{equation}
where $L_n(x) = \log x + 2\pi i n$,  $\im \log z \in (-\pi,\pi]$ for $z \in \mathbb C \setminus \{0\}$,
and the $c_{km}$ are combinatorial coefficients (see \cite[(4.20)]{wfcn}). The remaining solution $W_0(x)$ is the analytic function with power series $\sum_{k=1}^\infty (-k)^{k-1}x^k/k!$   (see \cite[(3.1)]{wfcn}). 

Below we will mostly use just the following simple consequence of \eqref{e:wn}.

\begin{lem}\label{l:lwf}
For $n\in \Integers \setminus\{0\}$, $\lim_{\ve \to 0} \re W_n(\ve) =-\infty$, 
$\lim_{\ve \uparrow 0}\im W_n(\ve )= \pi( 1+ 2n -\sgn n)$ and $\lim_{\ve \downarrow 0}\im W_n(\ve)=\pi(2n-\sgn n)$.
\end{lem}
\begin{proof} By \eqref{e:wn}, $\re W_n(\ve) \sim \log|\ve| \to -\infty$ as $\ve \to 0$. Meanwhile $\im L_n(\ve)  \to 2 \pi n + \im \log \ve$ and $\im \log(L_n(\ve)) \to \pi \sgn n$, as desired.
\end{proof} 

\subsection{Resonances}\label{s:resprelim} Let $V \colon \mathbb R^2\to\mathbb R$ be bounded and supported in $\{x \colon |x| \le \rho\}$, and let $P=-\Delta+V$. We discuss separately resonances in the cut plane $\mathbb C \setminus -i[0,\infty)$, at zero, and on the Riemann surface of the logarithm.

\subsubsection{Resonances on the cut plane}\label{s:rescut} Recall that in Section \ref{s:pereig} we defined a \textit{resonance} of $P$ to be a value of $\lambda$
with $-\pi/2 \leq \arg \lambda<3\pi/2$
for which there is an \textit{outgoing} solution $u \not \equiv 0$ to $(P-\lambda^2)u=0$, i.e. one obeying $u(r,\theta)=\sum_{\ell=-\infty}^\infty  c_\ell H_\ell^{(1)}(\lambda r)  e^{i\ell \theta}$ for $r > \rho$, where $V(x)=0$ for $|x|>\rho$.

\begin{lem}\label{l:uhprescut}
    If $\lambda  \in \mathbb C \setminus -i[0,\infty)$ is a resonance, then $\im \lambda < 0$ or $\re \lambda = 0$.
\end{lem}

\begin{proof}
If $\im \lambda>0$, then $u \in H^2(\mathbb R^2)$ by the  Hankel function asymptotics \eqref{e:hankellarge} and the equation $(P-\lambda^2)u=0$. Hence  $\int_{\mathbb R^2} |\nabla u|^2 + (V-\lambda^2)|u|^2 = 0$, implying that $\lambda^2$ is real and $\re \lambda =0$. By Rellich's uniqueness theorem \cite[Theorem~3.33]{dyatlovzworski}, there are no resonances in $\mathbb R \setminus \{0\}$.
\end{proof}

The proof of \cite[Theorem~3.33]{dyatlovzworski} uses the following characterization of resonances, which we will also use in Section \ref{s:lcr}. Let $R_V(\lambda) = (P-\lambda^2)^{-1}\colon L^2(\mathbb R) \to L^2(\mathbb R)$ be the resolvent, defined when $\im \lambda >0$ and $\lambda^2$ is not an eigenvalue. It is given by $R_V(\lambda) = R_0(\lambda) (I+V R_0(\lambda))^{-1}$, where $R_0(\lambda) = (-\Delta - \lambda^2)^{-1}$ is the free resolvent and has integral kernel  $\frac i 4 H_0^{(1)}(\lambda|x-y|)$ (see \cite[Section 2.1]{cdobs}). Using analytic Fredholm theory as in  \cite[Section 3.2.1]{dyatlovzworski} shows that if $\chi \in C_c^\infty(\mathbb R^2)$ and $\chi V = \chi$, then the cutoff resolvent $\chi R_V(\lambda)\chi$ continues meromorphically to $\mathbb C \setminus -i[0,\infty)$, and
\begin{equation}\label{e:rvr0}
\chi R_V(\lambda) \chi = \chi R_0(\lambda) \chi (I+V  R_0(\lambda)\chi )^{-1}.
\end{equation}
Moreover, the \textit{Birman--Schwinger principle} holds in the following form:  $\lambda$ is a pole of $\chi R_V(\lambda) \chi$ if and only if $-1$ is an eigenvalue of  $VR_0(\lambda)\chi$, i.e. if and only if $g = - V R_0(\lambda) g$ for some $g \not \equiv 0$. 

\begin{lem}\label{l:resequivcut} Let
$\lambda \in \mathbb C \setminus -i[0,\infty)$ and $\chi \in C_c^\infty(\mathbb R^2)$ with $\chi V = V$. Then $\lambda$ is a resonance if and only if $\chi R_V(\lambda) \chi$ has a pole at $\lambda$.
\end{lem}

\begin{proof}
Suppose $\chi R_V(\lambda) \chi$ has a pole at $\lambda$. Take $g$ as in the Birman--Schwinger principle and let $u = R_0(\lambda) g$. Then $(P-\lambda^2) u = 0$ and $u$ is outgoing in the sense of \cite[Definition 3.32]{dyatlovzworski}. The equivalence of that definition with ours above follows from the fact that, by variation of parameters, if $g(r,\theta) = g_\ell(r) e^{i\ell \theta}$ then $[R_0(\lambda) g] (r,\theta) = \frac {\pi i}2 \int_0^\infty J_\ell(\lambda \min(r,r')) H_\ell^{(1)}(\lambda \max(r,r')) g_\ell(r')\, dr' e^{i\ell \theta}$.

Conversely, suppose  $u \not \equiv 0$ is outgoing and $(P-\lambda^2)u=0$. Let $g= (-\Delta-\lambda^2) u$. Then $V R_0(\lambda) g = Vu = (\Delta + \lambda^2)u = -g$, so $\chi R_V(\lambda) \chi$ has a pole at $\lambda$ by the Birman--Schwinger~principle.
\end{proof}

\subsubsection{Resonances at zero} In Section \ref{s:pereig} we have already extended the notion of resonance and resonant state to $\lambda=0$, with definitions in terms of the $c_\ell$ of the expansion \eqref{e:outdef0}.  Equivalently, we may say that $P$ has an $s$-resonance at zero if $Pu=0$ has a bounded solution which is not $O(|x|^{-1})$ as $|x| \to \infty$, that $P$ has a $p$-resonance at zero if $Pu=0$ has a solution which is $O(|x|^{-1})$ but not $O(|x|^{-2})$, and that $P$ has an eigenvalue at zero if $Pu=0$ has a 
nontrivial solution which is $O(|x|^{-2})$.

\subsubsection{Resonances on the Riemann surface of the logarithm}\label{s:riemsurf} In Section \ref{s:lcr} we extend some of our results beyond the branch cut of  $\mathbb C \setminus -i[0,\infty)$, to $\Lambda$, the Riemann surface of the logarithm. We denote points on $\Lambda$ by $\lambda = |\lambda| e^{i \arg \lambda}$, but without identifying values of $\arg$ that differ by an integer multiple of $2\pi$. (Identification of such values of $\arg$ gives the universal covering $\Lambda \to \mathbb C \setminus \{0\}$.) Thus $\log \colon \Lambda \to \mathbb C \setminus \{0\}$ is biholomorphic, and  the functions $H^{(1)}_\ell(\lambda r)$ and $H^{(2)}_\ell(\lambda r)$ continue analytically from $\mathbb C \setminus -i[0,\infty)$ to $\Lambda$, where we identify $\mathbb C\setminus -i[0,\infty)$ with $\{\lambda \in \Lambda \colon -\pi/2 < \arg \lambda < 3\pi/2\}$.  Using 
\eqref{e:rvr0}, this implies that $\chi R_V(\lambda)\chi$
continues meromorphically to $\Lambda.$

The same definitions of resonances and outgoing functions from Sections \ref{s:pereig} and \ref{s:rescut} still work, and  Lemmas \ref{l:uhprescut} and \ref{l:resequivcut} extend directly:

\begin{lem}  Let
$\lambda \in \Lambda$ and $\chi \in C_c^\infty(\mathbb R^2)$ with $\chi V = V$. Then $\lambda$ is a resonance if and only if $\chi R_V(\lambda) \chi$ has a pole at $\lambda$. Moreover, there are no resonances in $\{\lambda \in \Lambda \colon 0<|\arg \lambda - \pi/2| \le \pi/2\}$. 
\end{lem}

\section{Eigenvalues and resonances} 

In this section we study eigenvalues and resonances near $0$,
proving Theorems \ref{t:persist} and \ref{t:lt}.  We begin with the circular well case, where $(-\Delta +V-\lambda^2)u=0$ can be solved explicitly using Bessel and Hankel functions.

\subsection{The circular well}\label{s:cwr}  For 
the circular well 
potential we can characterize eigenvalues and resonances as the zeros of certain functions defined using Bessel and Hankel functions.  Our main
tools are expansions near $0$ 
of such functions, Taylor approximations, and Rouch\'{e}'s Theorem.  Lemmas \ref{l:2c} and \ref{l:pwrc} together
 prove Theorem \ref{t:lt} for the 
special case $V_\ve= (-a_0^2+\ve)\bbo_\rho$. Combined with 
Lemma \ref{l:0mc} they prove Theorem \ref{t:persist} for
this $V_\ve$.

We consider the circular well potential 
$V=V(r)=-a^2\bbo_\rho$ given by \eqref{e:circwelldef} with
$a,\;\rho>0$.
We note that some of our results hold for other values of $a$, but since our main interest is tracking negative eigenvalues as they reach the continuous spectrum, it suffices to consider $a>0$. 
We define
$$\mu=\mu(\lambda, a)=\sqrt{\lambda^2+a^2}$$
where we take the square root to be positive for $\lambda>0$ and to depend continuously on $\lambda$.  By the Bessel  equation and asymptotics \eqref{e:bessdef}, \eqref{e:bessdiff}, \eqref{eq:Yexp}, if $u\in L^2_{\loc}(\Real^2) $ solves $(-\Delta +V-\lambda^2)u=0$, then 
\begin{equation}\label{eq:inside}
u(r,\theta)= \sum_{\ell=-\infty}^\infty b_{\ell} (u)J_\ell(\mu r)e^{i\ell \theta},\;\text{for}\; 0\leq r<\rho.
\end{equation}
Using that $u$ and $\partial_ru$ are continuous at $r=\rho$ this yields, when combined with \eqref{e:outdef}, 
\begin{equation}\label{eq:acd}
b_\ell =\frac{c_\ell  H^{(1)}_\ell(\lambda\rho)+d_\ell H^{(2)}_\ell(\lambda\rho)}{J_{\ell}(\mu \rho)}=\frac{\lambda}{\mu}\frac{c_\ell H^{(1)'}_\ell(\lambda\rho)+d_\ell H^{(2)'}_\ell(\lambda\rho)}{J_\ell '(\rho\mu)}.
\end{equation}
We note the following consequence if $u$ is nonvanishing
in the $\ell$th mode
\begin{equation}
    \label{eq:econd}
c_\ell(u)\not=0, d_\ell(u)=0 \Leftrightarrow 
\mu H_\ell^{(1)}(\rho \lambda)J'_{\ell}(\rho \mu)-\lambda H_\ell^{(1)\prime}(\rho \lambda) J_{\ell}(\rho \mu)=0
\end{equation}
for $\lambda \not =0$.
This shall be important for our study of
eigenvalues and resonances, so we define
\begin{align}
\label{eq:Qdef}Q_\ell(\lambda,  a)& := \mu H^{(1)}_\ell (\rho \lambda ) J'_\ell(\rho \mu)-\lambda H_\ell^{(1)\prime}(\rho \lambda) J_{\ell}(\rho \mu).
\end{align}
Although $Q$ depends on $\rho$ as well, we omit this
in our notation because we will not be varying $\rho.$ Using the recurrence relations \eqref{eq:recur} gives
\begin{equation}\label{eq:Ql-1}
Q_\ell(\lambda,a) = \mu J_{\ell-1}(\rho \mu)H_{\ell}^{(1)}(\lambda \rho) -\lambda J_\ell (\rho \mu) H_{\ell-1}^{(1)}(\lambda \rho),
\end{equation}
which for $\ell=0$ can also be written
\begin{equation}
Q_0(\lambda,a)=-\mu J_1(\rho \mu)H_0^{(1)}(\rho \lambda)+\lambda J_0(\rho \mu)H_1^{(1)}(\lambda \rho).
\end{equation}

Now we  allow the depth of the well to depend on 
$\ve$, so that 
\begin{equation}\label{eq:cwep}
a^2=a^2(\ve)=a_0^2-\ve, \; a_0>0, \; \text{ and set }V_\ve(r)= -a^2(\ve)\bbo_\rho (r).
\end{equation}
Notice that $\ve>0$ means the well is shallower than with
$\ve=0$, and $\ve<0$ means the well is deeper than for
$\ve=0.$  We remark that we shall assume $a_0>0$ throughout this section.
\begin{lem}\label{l:limit}  Let $a_\ve$ be as in \eqref{eq:cwep}, with $a_0>0$.
Suppose there is an $\ve_0>0$ so 
that for each $\ve\in (-\ve_0,0)$ there exists nontrivial
$\psi_\ve\in L^2(\Real^2)$, $\partial_\theta \psi_\ve=i\ell \psi_\ve$,
$(-\Delta -a^2_\ve\bbo_\rho -E_\ve)\psi_\ve=0$, and $E_\ve \uparrow 0$ as $\ve \uparrow 0$.  Then  $J_{|\ell|-1}(a_0\rho)=0$, and if $\ell=0$ then this means $J_1(a_0\rho)=0.$
\end{lem}
We shall later see that the implication goes the the other 
way as well.
\begin{proof}
Using the notation of \eqref{e:outdef} we see that we must
have $d_\ell(\psi_\ve)=0$ but $c_\ell(\psi_\ve)\not=0$.
Using \eqref{eq:econd} and  \eqref{eq:Qdef} we see that
this can happen if and only $Q_{\ell}(i\sqrt{|E_\ve|},a_\ve)=0$.

By \eqref{e:bessdef} and \eqref{eq:Yexp}, as $\lambda \rightarrow 0$ we have   $|H_{\ell}^{(1)}(\lambda)|\propto |\lambda|^{-\ell} $ for $\ell >0$ and $|H_0^{(1)}(\lambda)|\propto |\log \lambda|$.  Using these and the continuity of $J_{\ell}$, $J_{\ell-1}$  in \eqref{eq:Ql-1} shows that we must have 
$J_{\ell-1}(\rho \mu(0,a_0))=J_{\ell-1}(\rho a_0)=0$ if $\ell \ge 0$. The 
result for $\ell<0$ follows from the result for $|\ell|$.
\end{proof}

Lemma \ref{l:limit} has a natural interpretation in terms 
of a resonance or eigenvalue of $-\Delta-a_0^2\bbo_\rho$ at $0$.
\begin{lem} \label{l:resinlimit} For $a_0>0$ the operator $-\Delta-a_0^2\bbo_\rho$
\begin{itemize}
    \item has an $s$-resonance
if and only if $J_{1}(\rho a_0)=0$.
\item has a $p$-resonance 
if and only if $J_{0}(\rho a_0)=0$.
\item has an eigenvalue $0$ in the $\ell$th mode, $|\ell|\geq 2$, if and only if 
$J_{|\ell|-1}(\rho a_0)=0$.
\end{itemize}
\end{lem}
\begin{proof}
This follows by
using \eqref{e:outdef0}, \eqref{eq:inside}, the fact that 
the solution $u$ and $\partial_r u$ must be continuous across
$r=\rho$, and \eqref{eq:recur}.
\end{proof}

In the remainder of Section \ref{s:cwr} we prove a sequence of lemmas which, among other things, compute asymptotics of eigenvalues and resonances as $\varepsilon \to 0$. Figures \ref{f:3foldres} and \ref{f:ell=0} illustrate these results numerically.  Recalling that with $V$
real, $|\lambda|e^{i\arg \lambda}$ is a resonance of $-\Delta+V$ if and only if $|\lambda|e^{i(\pi -\arg \lambda)}$ is a resonance, we focus on the region with $|\arg \lambda|\leq \pi/2.$

Our first result is for the case in which 
$-\Delta+V_0$ has an $s$-resonance.  
\begin{lem}\label{l:0mc}
Let $a_0>0$ and suppose for $-\ve_1<\ve<0$ there exists a nontrivial radial function
$\psi_\ve\in L^2(\Real^2)$ , 
with $(-\Delta -a^2_\ve\bbo_\rho -E_\ve)\psi_\ve=0$, and $E_\ve \uparrow 0$ as $\ve \uparrow 0$.  Then there are  $\delta_0,\; \ve_0>0$ such that for $0<\ve<\ve_0$ the operator
$-\Delta -a_{\ve}\bbo_\rho$ has  no resonance in $\{\lambda\in\Complex: 0\leq |\lambda|<\delta_0,\; |\arg \lambda|\leq \pi/2\}$.  Moreover, as $\ve \uparrow 0$,
$E_{\ve}= -\frac{4}{\rho^2} e^{4/(\rho^2 \ve) - 2\gamma}(1+O(\ve))$, with $\gamma$ the Euler constant as in \eqref{eq:Yexp}.
\end{lem}
We see then that in the language of the introduction a negative eigenvalue corresponding to a rotationally symmetric eigenfunction (that is, one in the $\ell =0$ mode) does not persist.

\begin{proof}
From Lemma \ref{l:resinlimit} we see that $J_1(\rho a_0)=0$.  Thus we use Taylor polynomials to expand $Q_0(\lambda, a_{\ve})$ for small $|\lambda| $ and $|\ve|$ with $J_1(\rho a_0)=0$.  
As preliminaries,  note that 
$$\mu(\lambda, a_{\ve})= a_0+\frac{\lambda^2-\ve}{2a_0}-\frac{1}{8a_0^3}(\lambda^2-\ve)^2+O((\lambda^2-\ve)^3),$$ and, using $J_1(\rho a_0)=0$, \eqref{eq:recur} and the Bessel differential equation \eqref{e:bessdiff},
\begin{align*}
\frac{d}{d\mu}(\mu J_1(\rho \mu))\upharpoonright_{\mu=a_0}& = a_0\rho J_1'(\rho a_0)= a_0\rho J_0(\rho a_0),\\
\frac{d^2}{d\mu^2}(\mu J_1(\rho \mu))\upharpoonright_{\mu=a_0}& = 2\rho J_1'(\rho a_0)+a_0\rho^2J_1''(\rho a_0)= \rho J_0(\rho a_0).
\end{align*}
We note that here and below all errors are uniform in $|\arg \lambda|\leq \pi/2.$
Using these to expand $\mu J_1(\rho \mu)$ in $\lambda$ and $\ve$ yields
\begin{align*}
\mu J_1(\rho \mu) &= a_0\rho J_0(\rho a_0) \left( \frac{\lambda^2-\ve}{2a_0}-\frac{1}{8a_0^3}(\lambda^2-\ve)^2\right) + \frac{\rho}{2} J_0(\rho a_0)\left( \frac{1}{2a_0}(\lambda^2-\ve)\right)^2 + O((\lambda^2-\ve)^3)\\ & = -\frac{1}{2}J_0(\rho a_0) \rho \ve +O((\lambda^2-\ve)^3).
\end{align*}
This, along with the small $\lambda$ asymptotics of $H^{(1)}_0(\rho \lambda)$ and 
$H^{(1)}_1(\rho \lambda)$ from \eqref{e:bessdef} and \eqref{eq:Yexp}, gives
\begin{align*}
Q_0(\lambda,  a_{\ve}) =& -\left( -\frac{1}{2}J_0(\rho a_0)\rho \ve + O((\lambda^2-\ve)^3) \right)\left(1+\frac{2i}{\pi}\log \left(\frac{\lambda \rho}{2}\right)  +\frac{2i}{\pi}\gamma+O(\lambda^2 \log \lambda)\right) \\ &+ i\lambda \left(J_0(\rho a_0)-J_0'(\rho a_0) \frac{\rho\ve}{2a_0}+O(\lambda^2)+O(\ve^2)\right)\left(-\frac{2}{\pi \lambda \rho}+O(\lambda \log \lambda)\right).
\end{align*}
Using \eqref{eq:recur}, $J_0'(\rho a_0)=0$.
Now set 
\begin{equation}\label{eq:g0c}
g(\lambda, a_0,\ve)= J_0(\rho a_0)\left( \frac{\rho }{2}\ve \left(1+\frac{2i}{\pi}\log \left(\frac{\lambda \rho}{2}\right)+\frac{2i}{\pi}\gamma \right) -\frac{2i}{\pi \rho}\right)\end{equation}
which is chosen so that
\begin{equation}\label{eq:Q0g0}
Q_0(\lambda,  a_{\ve})=g(\lambda, a_0,\ve)+ O(\lambda^2 \log \lambda)+O(\ve^3 \log \lambda)+ O(\ve^2).
\end{equation}
Set $\tilde{\lambda}_\ve= \frac{2}{\rho}\exp( \frac 2 {\rho^2 \ve} -\gamma + \frac{i\pi}2)$ and note that 
\begin{equation}\label{eq:0modesol}
g(\lambda, a_0,\ve)=0 \Leftrightarrow \lambda =\tilde{\lambda}_\ve.
\end{equation}
We have $\lim_{\ve \uparrow 0} \tilde{\lambda}_\ve=0$ but $\lim_ {\ve \downarrow 0} |\tilde{\lambda}_\ve|=\infty$.

First consider $\ve<0$.
For $w\in \Complex$ with $|w|$ small, 
\begin{equation}\label{eq:ges}
g(\tilde{\lambda}_\ve (1+w),a_0,\ve) =i J_0(\rho a_0)\frac{\rho}{\pi} w\ve(1+O(w)).\end{equation}  Since $\ve\log \tl_\ve=O(1)$ and $\tl_\ve=O(e^{2/(\rho^2\ve)})$, 
for $\ve<0$ with $|\ve|$ sufficiently small we can apply Rouch\'{e}'s Theorem to the pair
$Q_0(\lambda, a_\ve)$ and $g(\lambda, a_0,\ve)$  on the circle 
with center $\tl_\ve$ and radius $|\tl_\ve||\ve|^{1/2}$
to see that there is an $\ve_0>0$ so that $Q_0(\lambda, a_\ve) $ has a single simple zero at 
$\lambda_\ve$ satisfying $\lambda_\ve -\tilde{\lambda}_\ve=O(\ve^{1/2} \tilde{\lambda}_{\ve})$ when $-\ve_0<\ve<0$.
Writing  $\lambda_\ve=\tl_\ve(1+f)$, using \eqref{eq:Q0g0}
and \eqref{eq:ges}
yields
$$0=Q_0(\lambda_\ve,a_0,\ve)= g(\tl_\ve(1+f),a_0,\ve)+ O(\ve^2)= \ve J_0(\rho a_0)\frac{i\rho}{\pi} f +O(\ve f^2)+O(\ve^2). $$
Since $f=O(\ve^{1/2})$, this shows $f=O(\ve).$
Since $0<\arg \lambda_\ve<\pi,$  $\lambda_\ve^2$ is an eigenvalue of $-\Delta -a_\ve^2\bbo_\rho.$

If $\ve>0$ and  $|\arg\lambda|\le \pi/2$, then $|g(\lambda, a_0,\ve)|\geq  
|J_0(\rho a_0)|\left(\frac{2}{\pi \rho}-\ve \frac{\rho}{\pi }\re \log \lambda \right) +O(\ve)$, so that
$$|Q_0(\lambda,  a_{\ve})|\geq |J_0(\rho a_0)|\big(\tfrac{2}{\pi \rho}-\ve \tfrac{\rho}{\pi }\re \log \lambda \big) +O(\ve) + O(\ve^3 \log \lambda) +O(\lambda^2 \log \lambda).$$
Hence, shrinking $\ve_0$ if necessary,
there is $\delta_0>0$ so that $Q_0(\lambda,  a_{\ve})\not =0$ for $0<\ve<\ve_0$ and $\{ \lambda: \; |\lambda|<\delta_0, \; |\arg \lambda|\leq \pi/2\}$, giving a
resonance-free region.
\end{proof}
We remark that the only place in which the assumption on the eigenvalue $E_\ve$ and eigenfunction of 
$-\Delta-a_\ve^2\bbo_\rho$ was used was to show $J_1(\rho a_0)=0$.  Starting with that assumption, the same proof shows that $-\Delta -a_\ve^2\bbo_\rho$ has an eigenvalue
increasing to $0$ as $\ve \uparrow 0$, with expansion as given.  This is also true (with obvious modifications) for the proofs of Lemmas \ref{l:2c} and \ref{l:pwrc}.

Next we turn to the case corresponding to $-\Delta +V_0$ having eigenvalue $0$.  
Here we will focus on eigenvalues and on
resonances in the fourth quadrant, but in Section \ref{s:lcr}
we will show that
for $\ve$ having either sign there are many more 
resonances related to this eigenvalue at $0$.
\begin{lem}\label{l:2c}
Let $|\ell| \geq 2$, and suppose  for $-\ve_1<\ve<0$ there exist nontrivial
$\psi_\ve\in L^2(\Real^2)$,  $\partial_\theta \psi_\ve=i\ell \psi_\ve$
with $(-\Delta -a^2_\ve\bbo_\rho -E_\ve)\psi_\ve=0$, and $E_\ve \uparrow 0$ as $\ve \uparrow 0$.  Then there is an $\ve_0>0$ so that for $0<|\ve|<\ve_0$, $-\Delta -a_\ve^2 \bbo_\rho$ has a resonance satisfying $\lambda =\sqrt{\ve (|\ell| -1)/|\ell|  }+O(\ve^{3/2}(1+\delta_{|\ell|,2}\log |\ve|)).$
Moreover, as  $\ve \uparrow 0$, $E_\ve = \frac{|\ell| -1}{|\ell |}\ve+ O(\ve^2(1+\delta_{|\ell|,2}\log |\ve|))$.
\end{lem}

In the language of the introduction, in this case the eigenvalue
persists. For $\ve>0$ the resonance is near the positive real axis, with negative imaginary part by Lemma \ref{l:uhprescut}. 

\begin{proof}  We give the proof for $\ell\geq 2$, as the result for $-\ell$ follows from that for $\ell$.

By Lemma \ref{l:limit}  we have $J_{\ell -1}(\rho a_0)=0$.  
As in the proof of Lemma \ref{l:0mc} we expand $Q_\ell$ for $\lambda$ and $\ve $ near $0$.   By  \eqref{eq:recur} and \eqref{e:bessdef},  $J_{\ell -1}'(\rho a_0)=-J_{\ell}(\rho a_0)$
 and $H^{(1)}_\ell(\rho \lambda)=iY_\ell(\rho \lambda)+O(\lambda^\ell)$.  Using this, Taylor expanding $\mu$ and $J_{\ell -1}(\mu \rho)$, and using 
  the expansions \eqref{eq:Yexp} 
for $Y_{\ell}$, $Y_{\ell -1}$ yields
\begin{equation}\begin{split}\label{eq:Qellexp}
 Q_\ell(\lambda,&a_\ve) = i\left\{ \left( -J_\ell(\rho a_0)a_0\rho\frac{\lambda^2-\ve}{2a_0}+O((\lambda^2-\ve)^2 )\right) \left( -\frac{1}{\pi}\left( \frac{\rho \lambda}{2}\right)^{-\ell} (\ell -1)!
+O(\lambda^{-\ell+2})\right)\right.  \\
& -\left. \lambda \left(J_\ell(\rho a_0)+O(\lambda^2-\ve)\right)\left( -\frac{1}{\pi}\left( \frac{\rho \lambda}{2}\right)^{-\ell +1}(\ell -2)! +O(\lambda^{-\ell+3}(1+\delta_{\ell,2}\log \lambda))\right)\right\}. 
\end{split}\end{equation}
Now set 
\begin{equation}\label{eq:g1l}
g(\lambda, a_0,\ve)= \frac{i\rho}{2\pi}\left( \frac{\rho \lambda}{2}\right)^{-\ell}(\ell-2)! J_{\ell}(\rho a_0)\{ (\lambda^2-\ve)(\ell -1) +\lambda^2\}
\end{equation}
which is chosen so that 
\begin{equation}\label{eq:Qlg}
Q_\ell(\lambda,a_\ve)= g(\lambda, a_0,\ve) + O(\lambda^{-\ell+4}(1+\delta_{2,\ell}\log \lambda))+O(\ve^2\lambda^{-\ell})+ O(\ve \lambda^{-\ell+2}).
\end{equation}
Set $\tilde{\lambda}_\ve= \sqrt{\ve(\ell -1)/\ell}$, with the branch of square root taken such that $\re \tilde{\lambda}_\ve \geq 0$ and
 $\im \tilde{\lambda}_\ve \geq 0$, and note that  $g(\lambda,a_0,\ve)=0$ if and only if $\lambda = \pm \tilde{\lambda}_\ve.$ 
There is a $C_0>0$ so that for $|w|$ small, 
 $|g(\pm \tilde{\lambda}_\ve+ w,a_0,\ve)|\geq   C_0|\tilde{\lambda}_\ve|^{-\ell +1}|w|+O(|w|^2|\tilde{\lambda}_\ve|^{-\ell})$. Since \eqref{eq:Qellexp} implies that $|\tilde{\lambda}_\ve|\propto |\ve|^{1/2}$, 
 we can ensure that for
$|\ve|$ sufficiently small 
$$|Q_\ell(\tilde{\lambda}_\ve+ \ve z,a_\ve)- g( \tilde{\lambda}_\ve+ \ve z)|<|g(\tilde{\lambda}_\ve+ \ve z, a_0,\ve)|\; \text{for $|z|=1$}.$$
Hence  an application of Rouch\'{e}'s Theorem  on the circle centered at $\tilde{\lambda}_\ve$ with radius $|\ve|$ shows that there is an $\ve_0>0$
so that for 
$0<|\ve|<\ve_0$, $Q_\ell(\lambda,a_\ve)$ has a single
simple zero  $\lambda_\ve$ within $|\ve|$ of $\tilde{\lambda}_\ve$.   Note that for $\ve \not =0$ small enough the disk enclosed by this circle stays away from $\lambda=0$.  This zero of 
$Q_\ell$ at $\lambda_\ve$ is a resonance of 
$-\Delta +V_\ve$. 

The error can be improved  using  \eqref{eq:g1l} and \eqref{eq:Qlg} as we describe.  Write $\lambda_\ve =\tilde{\lambda}_\ve+f_\ve$.  Then
$$0=\lambda^{\ell}_\ve Q_\ell(\lambda_\ve, a_\ve)= \frac{i\rho}{\pi}\left( \frac{\rho}{2}\right)^{-\ell}(\ell-2)! J_{\ell}(\rho a_0)\ell \tilde{\lambda}_\ve f_\ve + O(f_\ve^2) + O(\ve^2(1+\delta_{\ell,2}\log \ve))$$
using $|\tilde{\lambda}_\ve|\propto |\ve|^{1/2}$.  Solving for $f_\ve$ and using $f_\ve=O(\ve)$ yields $f_\ve=O(\ve^{3/2} (1+\delta_{2,\ell}\log \ve))$.

For $\ve <0$, $\lambda_\ve\in i(0,\infty)$, and hence $\lambda_\ve^2$ is a negative eigenvalue of $-\Delta-a^2_\ve\bbo_\rho$. 
\end{proof}


We now turn to a second kind of persistent eigenvalue.  
\begin{lem}\label{l:pwrc}
Suppose  for $-\ve_1<\ve<0$ there exist nontrivial
$\psi_\ve\in L^2(\Real^2)$ ,  $\partial_\theta \psi_\ve(r,\theta)=\pm i\psi_\ve$, 
with $(-\Delta -a^2_\ve\bbo_\rho -E_\ve)\psi_\ve=0$, and $E_\ve \uparrow 0$ as $\ve \uparrow 0$.  Then there is an $\ve_0>0$ so that for $0<|\ve|<\ve_0$, $-\Delta -a_\ve^2 \bbo_\rho$ has a resonance satisfying 
$$\lambda =\frac{2i}{\rho}\exp\left(\frac{1}{2}-\gamma+ \frac{W_{-1}(\ve\rho^2 e^{-1+2\gamma}/4)}{2}\right)+o(|\ve|^{3/2}).$$
Moreover, as  $\ve \uparrow 0$ 
$$E_\ve = -\frac{4}{\rho^2}\exp\left( 1-2\gamma +W_{-1}(\ve \rho^2 e^{-1+2\gamma}/4)\right)+o(\ve^2)=-\frac{\ve}{W_{-1}(\ve \rho^2 e^{-1+2\gamma}/4)} + o(\ve^2 ).$$
\end{lem}
\begin{proof}We give the proof for $\ell=1$, as  the case $\ell =-1$ then follows.
The difference between this case and that of Lemma \ref{l:2c}, with $\ell \geq 2$,
 is that the leading term in the expansion of $H^{(1)}_0(\lambda)=H^{(1)}_{\ell-1}(\lambda)$ behaves like $\log \lambda$ rather than
$\lambda^{-\ell -1}$.  Using $J_0(\rho a_0)=0$ from Lemma 
\ref{l:limit}  gives (compare \eqref{eq:Qellexp}):
\begin{align}\label{eq:Q1}
 Q_1(\lambda,a_\ve) = &i \left( -J_1(\rho a_0)a_0\rho\frac{\lambda^2-\ve}{2a_0}+O((\lambda^2-\ve)^2 )\right) \left( -\frac{1}{\pi}\left( \frac{\rho \lambda}{2}\right)^{-1} 
+O(\lambda \log \lambda) \right) \nonumber  \\
& - \lambda \left(J_1(\rho a_0)+O(\lambda^2-\ve)\right)\left(    1+ \frac{2i}{\pi}\log \left( \frac{\lambda \rho  }{2}\right) +\frac{2i}{\pi}\gamma +O(\lambda^2 \log \lambda)   \right). 
\end{align}
Set
\begin{equation}\label{eq:g1}g(\lambda,a_0,\ve)= \frac{2 \lambda}{i\pi}J_1(\rho a_0)\left(\log \left(\frac{\lambda \rho}{2}\right)-\frac{1}{2}+\frac{\pi}{2i}+\gamma +\frac{\ve}{2\lambda^2}\right)
\end{equation}
so that 
\begin{equation}\label{eq:Q1g}Q_1(\lambda, a_\ve)=g (\lambda,a_0,\ve)+ O(\lambda^3 \log \lambda)+ O(\ve^2 \lambda^{-1}) + O(\ve \lambda \log \lambda).
\end{equation}
Set $b= -\log (\rho/2)+1/2+i\pi/2-\gamma$ and $\tau =\log \lambda -b$.  Then $\lambda$ satisfies $g(\lambda,a_0,\ve)=0$ if and only if 
$2 \tau e^{2\tau}=-\ve e^{-2b}$.  Solutions of this equation are given by 
$$\tau = \frac{1}{2}W_n(-\ve e^{-2b})= \frac{1}{2} W_n\left(\ve \frac{\rho^2}{4} e^{-1+2\gamma}\right).$$
The case $n=0$ is not relevant for us, as $\lim _{\ve \rightarrow 0} W_0(\ve)=0$.
We want $-\pi/2 <\im \log \lambda=\im \tau +\pi/2 \leq \pi/2 $ when $\ve$ is near $0$.  By  Lemma \ref{l:lwf} this holds if and only
if $n=-1$.  Set
\begin{equation}\label{eq:lwwf}
\tilde{\lambda}_\ve =\frac{2i}{\rho} \exp\left( \frac{1}{2}- \gamma + \frac{1}{2} W_{-1}\left(\ve \frac{\rho^2}{4} e^{-1+2\gamma}\right)\right),
\end{equation}
then $g(\tilde{\lambda}_\ve, a_0,\ve)=0$
and $\tilde{\lambda}_\ve\rightarrow 0$ as $\ve \rightarrow 0$ 
by Lemma \ref{l:lwf}.
We use $g'(\tilde{\lambda}_\ve) \propto \log \tilde{\lambda}_\ve $ to apply Rouch\'{e}'s Theorem on the disk with center $\tilde{\lambda}_\ve$ and radius $|\ve|^{1/2}|\tilde{\lambda}_\ve|$ to the pair $Q_1$ and $g$.
This shows that for $|\ve|$ sufficiently small $Q_1(\lambda,a_{\ve})$  has a zero at $\lambda_\ve$
with $|\lambda_\ve-\tilde{\lambda}_\ve|< |\ve|^{1/2}|\tilde{\lambda}_\ve|$.  Using 
\eqref{eq:g1} and \eqref{eq:Q1g} allows us to improve this to an error of size $o(|\ve|^{3/2})$.

When $\ve<0$ $\arg \lambda_\ve =\pi/2$ and $\lambda_\ve^2$ is an eigenvalue.  When $\ve>0$, $-\pi/2<\arg \lambda_\ve<0$ and   $\lambda_\ve$ is near the positive real axis and is a resonance.
\end{proof}
We will see in Lemma \ref{l:pwrg} that for  any $n\in \Integers \setminus \{0\}$, $ \frac{2i}{\rho} \exp\left( \frac{1}{2}- \gamma + \frac{1}{2} W_n\left(\ve \frac{\rho^2}{4} e^{-1+2\gamma}\right)\right)$ 
gives the approximate location of a resonance on $\Lambda$ for sufficiently small $|\ve|$.

\subsection{More general potentials and resonances on the Riemann surface of the logarithm}\label{s:lcr}

In this section we turn to more general Schr\"{o}dinger operators and consider resonances and eigenvalues near $0$ but with the resonances on $\Lambda$ rather than on $\mathbb C \setminus -i[0,\infty)$; this means allowing more general values of $\arg \lambda$ than we did in Section \ref{s:cwr}.


Lemmas \ref{l:evpert} and \ref{l:pwrg} show that if 
$-\Delta+V_0$ has either a $p$-resonance or an eigenvalue
at $0$, then for small $|\ve|\not =0$, $-\Delta+V_\ve$ has many nearby resonances.
Figure \ref{f:moreres} illustrates some resonances for a family of
circular wells $V_\ve(r)=(-j_{0,1}^2+\ve)\bbo_1(r)$.  The  
subset of $\Lambda$ shown is larger than that shown in the illustrations of the introduction.
Here the branch cut is the brown ray in the second quadrant, so
that the string of resonances shown in the second quadrant (for varying values of $\ve$) has
argument between $-4\pi/3 $ and $-\pi$.  The color-coding here indicates the sign of $\ve$.  These resonances correspond to
the $\ell=1$ (or $\ell =-1$) mode and are the ones approximated in Lemma~\ref{l:pwrg}. 

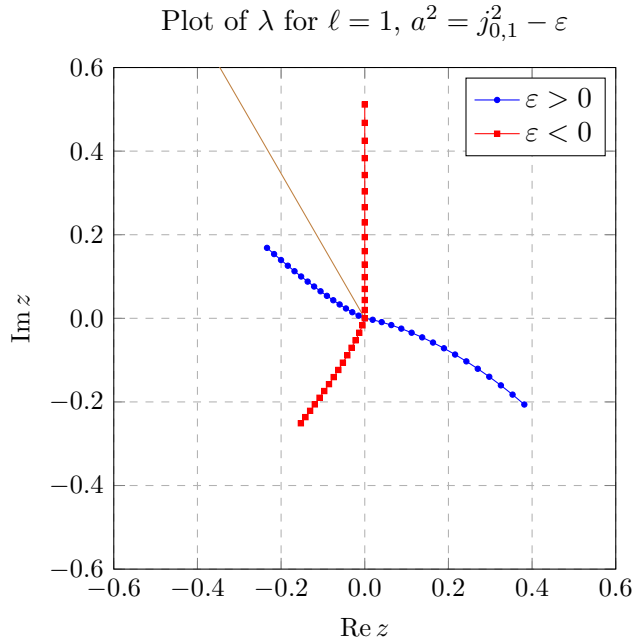
\begin{figure}[ht]
  \centering
  \begin{tikzpicture}
\begin{axis}[
y tick label style={
        /pgf/number format/.cd,
            fixed,
            fixed zerofill,
            precision=1,
        /tikz/.cd
    },
    x tick label style={
        /pgf/number format/.cd,
            fixed,
            fixed zerofill,
            precision=1,
        /tikz/.cd
    },
title={Plot of $\lambda$ for $\ell=1$, $a^2 = j_{0,1}^2 - \varepsilon$},
ylabel=\small{$\operatorname{Im} z$},
xlabel=\small{$\operatorname{Re} z$},
          grid=major, 
          grid style={dashed,black!30}, 
tick label style={font=\small}, 
width=0.5\linewidth, 
height=0.5\linewidth,
xmin = -0.6, xmax = 0.6,
ymin = -0.6, ymax = 0.6,
]
 \draw[color=brown] (0,0) -- (-0.57735,1);
\addplot+[mark options={scale=0.5}] table [x=RePos, y=ImPos, col sep=comma] {dataset1split.csv};
\addlegendentry{$\varepsilon>0$}
\addplot+[mark options={scale=0.5}] table [x=ReNeg, y=ImNeg, col sep=comma] {dataset1split.csv};
\addlegendentry{$\varepsilon<0$}
\end{axis}
\end{tikzpicture}
\caption{Plot of $\lambda_\ve$ in a sheet of $\Lambda$ cut along the brown ray, with arguments ranging from $-4\pi/3$ to $2\pi/3$, values of $\varepsilon$ as in Figure \ref{f:3foldres}.}\label{f:moreres}
\end{figure}

Recall that on $\Lambda$ we do not identify points with arguments that differ by a  multiple of $2\pi$.

In this section we show that by using results about the 
structure of  $\Rvz(\lambda)$ near $\lambda=0$ we can
obtain similar, but more general results, to those obtained in Section \ref{s:cwr}.  In addition to allowing more general potentials $V_\ve$, we  also show that in the case of the persistent eigenvalues there are, for small 
$|\ve|$, many associated resonances on $\Lambda$.  For simplicity we confine ourselves to the case of 
radial potentials.  However, at the 
end of this section we indicate how our proofs work in 
more generality, provided the singularity of the resolvent
$R_{V_0}$ at $\lambda =0$ has rank one.

We make the following assumptions throughout this section.
\begin{as}\label{a:Vve}
There is an $\ve_*>0$ such that 
$$(-\ve_*,\ve_*)\ni \ve \mapsto V_\ve\in L^\infty_c(\Real^2;\Real)$$ is $C^\infty$ in $\ve$, and $V_\ve$
is a radial potential  for each $\ve$.
\end{as}
These potentials are more general than those of 
Theorem \ref{t:lt}, since there is no assumption on the sign of 
$V_0$.  Our replacement is that we shall need a certain integral to be nonzero.
Denote  
\[\partial_\ve V_\ve= \dot{V}_\ve, \qquad  \dot{V}_0=\dot{V}_\ve\upharpoonright_{\ve=0}, \qquad \ddot{V}_0=\partial_\ve^2V_\ve\upharpoonright_{\ve=0},\]
and note that in the circular well example considered above in \eqref{eq:cwep} we have $\dot V_0 (x) = \bbo_\rho(|x|)$ and $\ddot V_0 = 0$.

It is easy to see  using \eqref{eq:rid}
that if $\| \chi R_{V_0}(i\kappa)\chi\|$
is bounded as $\kappa \downarrow 0$ for all $\chi \in C_c^\infty(\Real^2)$, then there are $\ve_0,\delta_0>0$ so
that for $|\ve|<\ve_0$, $-\Delta+V_\ve$ has no negative eigenvalues with norm less than $\delta_0^2.$ 
It is well known (see e.g. \cite[Theorem 1]{cdgen} or \cite{bgd}) that if 
there is a $\chi \in C_c^\infty(\Real^2)$ so that 
$\lim_{|\kappa \downarrow 0}\| \chi R_{V_0}(i\kappa)\chi\| =\infty$
then $0$ must be an eigenvalue, an $s$-resonance, or
a $p$-resonance of $-\Delta +V_0$.  
Hence our hypotheses for each lemma will involve the
existence of a certain kind of zero resonance or a zero eigenvalue of $-\Delta +V_0$; compare Lemma \ref{l:resinlimit}.  A consequence  of each of Lemmas \ref{l:srs}, \ref{l:evpert} and \ref{l:pwrg}
is
that then $-\Delta+V_\ve$ has negative eigenvalue approaching $0$ as $\ve$ goes to $0$ from one direction.
These
observations together with Lemmas \ref{l:evpert} and \ref{l:pwrg}  prove Theorem
\ref{t:lt}.  In fact, they prove much more: for $\ve\not =0$ and $|\ve|$ sufficiently small  there are many resonances related to the negative eigenvalue.  Combining this with Lemma \ref{l:srs} proves Theorem \ref{t:persist}. 

Choose $\chi \in C_c^\infty(\Real^2;[0,1])$ radial and such that $\chi V=V$, and note that
$$(-\Delta+V_\ve -\lambda^2)\Rvz(\lambda)\chi = \chi(I+(V_\ve-V_0)\Rvz(\lambda)\chi).$$
Hence for $\im \lambda>0$
\begin{equation}\label{eq:rid}
\Rve(\lambda)\chi = \Rvz(\lambda)\chi (I+(V_\ve-V_0)\Rvz(\lambda)\chi)^{-1}
\end{equation}
and the identity extends to $\Lambda$ by meromorphic continuation. As in Section \ref{s:resprelim}, we have the following version of the Birman--Schwinger principle: $\Rve$ has a pole at $\lambda$, a regular point
of $\Rvz$, if and only if $-1$ is an eigenvalue of $(V_\ve-V_0)\Rvz(\lambda)\chi$.

For $\ell \in \Integers $ define $\proj_\ell:L^2(\Real^2)\rightarrow L^2(\Real^2)$ by 
$$(\proj_\ell f)(r,\theta)=\frac{1}{2\pi} \int_{0}^{2\pi}f(r,\theta')e^{-i\ell \theta'}d\theta' e^{i\ell \theta}.$$
Since $V_\ve$ is radial, \eqref{eq:rid} yields another Birman--Schwinger type principle: $\Rve \proj_\ell$ has a pole at $\lambda$, a regular point of 
$\Rvz\proj_\ell$, if and only if $-1$ is an eigenvalue of 
$(V_\ve-V_0)\Rvz(\lambda)\proj_\ell\chi$. We will use expansions of $\Rvz(\lambda)$ from \cite{cdgen} to iterate this argument and replace $(V_\ve-V_0)\Rvz(\lambda)\proj_\ell\chi$ by  a suitable rank one operator $A(\lambda,\varepsilon)$, and study the zeroes of $\det(I+A(\lambda,\varepsilon)$).

Our first result in this setting parallels Lemma \ref{l:0mc}.  
\begin{lem}\label{l:srs}  Let $V_\ve$ satisfy Assumption \ref{a:Vve}.
Suppose $-\Delta +V_0$ has an $s$-resonance with corresponding resonant state $\psi(r,\theta)$ satisfying $\psi(r,\theta)=1/\sqrt{2\pi}$ for $r$ sufficiently large.
Suppose $\alpha_0=\int_{\Real^2} \dot{V}_0(r)|\psi(r)|^2 dx \not =0$.
There is an $\ve_0>0$ such that if $0<-\alpha_0\ve<|\alpha_0|\ve_0$, then $-\Delta +V_\ve$ has an eigenvalue at $- e^{2/(\ve \alpha_0)- 2\alpha_1/\alpha_0^2}(1+O(\ve))$
for some $\alpha_1\in \Real$, and the corresponding eigenfunction is rotationally symmetric.  Moreover, given $\varphi>0$, there are $\ve_0'(\varphi)>0$, $\lambda_0>0 $ such
that if  $0<\alpha_0\ve<|\alpha_0|\ve_0'$, then for
any $\tilde{\chi} \in C_c^\infty(\Real^2)$, $\tilde{\chi} \Rve(\lambda)\proj_0\tilde{\chi} $ is bounded for $|\arg \lambda|<\varphi$, $|\lambda|<\lambda_0$.
\end{lem}
The proof of the lemma gives an expression for $\alpha_1.$
Likewise, the  proofs of Lemmas \ref{l:evpert} and \ref{l:pwrg} give expressions for the constants $\alpha_1$
appearing in their statements.
\begin{proof}
By the Birman-Schwinger principle discussed above,   $\Rve\proj_0$ has a pole at $\lambda$, a regular point of
$R_{V_0}\proj_0$, if and only if $-1$ is an eigenvalue of 
$(V_\ve-V_0) R_{V_0}(\lambda)\proj_0\chi$. 
From the proof of  \cite[Theorem 2]{cdgen},  since $-\Delta+V_0$ has an $s$-resonance there is an operator $B_{00}:L^2_c\rightarrow H^2_{\loc}$ such that 
\begin{equation}\label{eq:res}
\chi R_{V_0}(\lambda)\proj_0 \chi = - \chi \psi \otimes \psi \chi \log \lambda  + \chi B_{00}\proj_0 \chi + O( \lambda^2 (\log \lambda)^3),
\end{equation}
where $u \otimes v$ denotes the operator taking $w$ to $\langle w, v\rangle u$ and $\chi\in C_c^\infty(\Real^2)$ satisfies $\chi V=V$.
Set $\tilde{R}_{V_0}(\lambda)= R_{V_0}(\lambda) +   \psi \otimes \psi  (\log \lambda -i\pi/2).$  The operator
$\tilde{R}_{V_0}(\lambda)\proj_0:L^2_c(\Real^2)\rightarrow L^2_{\loc}(\Real^2)$ is bounded at $\lambda=0$, 
and is chosen to make it easy to see that
the constant $\alpha_1$ below  is real.
  For $\varphi>0$ there are $\lambda_0(\varphi),\;\ve'_0(\varphi)>0$  such that in the region
$|\arg \lambda|<\varphi$,
$|\lambda|<\lambda_0$, $|\ve|<\ve_0'$, $I+(V_\ve-V_0)\tilde{R}_{V_0}(\lambda) \proj_0\chi$ is invertible, yielding
\[
(I +(V_\ve-V_0) R_{V_0}\proj_0\chi)  (I +(V_\ve-V_0)\tilde{R}_{V_0}\proj_0 \chi)^{-1}
=I + A(\lambda,\ve), \qquad \text{where}\]
\[
A(\lambda,\ve) := - (V_\ve-V_0) \psi \otimes \psi \chi (\log \lambda -i\pi/2) (I +(V_\ve-V_0) \tilde{R}_{V_0}\proj_0 \chi)^{-1}.
\]
Thus, the poles of $R_{V_\ve}(\lambda)\proj_0$ for $\lambda, \; \ve$ as described above equal the zeroes of $Q(\lambda,\ve):= \det(I+A(\lambda,\ve))$.
Using   $\det(I+u \otimes v T )=1+\tr u \otimes v T = 1+  \langle Tu , v \rangle$ gives
\begin{equation}\label{eq:cat}
Q(\lambda,\ve)=1 -\ve(\log \lambda -i\pi/2) (\alpha_0+ \alpha_1 \ve + O(\ve^2)+ O(\ve \lambda^2 (\log \lambda)^3),
\end{equation}
where $\alpha_1= -\langle (-\frac{i \pi }{2}\psi \otimes \psi +B_{00})\dot{V}_0\psi, \dot{V}_0\psi \rangle+ \frac 12\langle \ddot{V}_0\psi,\psi\rangle  $.   Since for $\kappa>0$ and $\chi \in C_c^\infty(\Real^2;\Real)$, 
$ \chi R_{V_0}(i\kappa)\chi$  is self-adjoint we find from the expansion \eqref{eq:res} that 
$\chi\left( -\frac{i\pi}{2}\psi \otimes \psi +B_{00}\right)\chi$
is self-adjoint, implying that $\alpha_1$ is real.

The rest of the proof closely resembles that of Lemma \ref{l:0mc}. Suppose first that $\alpha_0\ve>0$.  Then, since 
 $\re \log \lambda<0$ for $\lambda$ near $0$, it is easy to see from \eqref{eq:cat} that, shrinking $\ve_0'$, $\lambda_0$ if necessary, 
   $Q(\lambda,\ve)$ has no zeros in $ |\arg \lambda|<\varphi$, $|\lambda|<\lambda_0,$  $0<\alpha_0 \ve<|\alpha_0|\ve_0'$, and hence the cutoff resolvent is bounded.  
   
  Now consider
the case  $\alpha_0\ve<0$.  
Define $$g(\lambda, \ve):= 1- \ve(\log \lambda -i\pi/2) (\alpha_0+ \alpha_1 \ve),$$ and note $g(\lambda,\ve)=0$ if and only
if $\lambda=\tl_\ve$, where  $\log \tl_\ve =i\pi/2+\frac{1}{ \ve \alpha_0+\ve^2 \alpha_1}$.  Thus
$|\tl_\ve|=e^{1/\alpha_0\ve} e^{-\alpha_1/\alpha_0^2}(1+O(\ve))$ and $\arg \tl _\ve=\pi/2$. 
Using \eqref{eq:cat}, we can 
apply Rouch\'{e}'s theorem
on the circle $|\lambda - \tl_\ve|=|\tl_\ve||\ve|^{1/2}$  to see that 
if $0<-\alpha_0\ve$ is sufficiently small then $Q(\lambda,\ve)$ has a zero at a single point $\lambda_\ve$ satisfying $\log \lambda_\ve = \log \tl_\ve+O(\ve^{1/2})$.  This 
zero is a resonance of $-\Delta +V_\ve$.  Using the expansion \eqref{eq:cat} improves the error to $\log \lambda_\ve= \log \tl_\ve +O(\ve^2)$.  
     Since  $\lambda_\ve$ is in the physical space $\{\lambda \in \Lambda \colon 0< \arg \lambda <\pi\}$, $\lambda^2_\ve$  corresponds to
an eigenvalue of $-\Delta +V_\ve$.
\end{proof}
To compare this result with Lemma \ref{l:0mc}, we note that 
if $V_0(r)=-a_0^2\bbo_\rho(r)$, with $J_1(a_0\rho)=0$ and $\dot{V}_0=\bbo_{\rho}$, then, in 
the notation of Lemma \ref{l:srs} 
$\psi(r)=(\sqrt{2\pi} J_0(\rho a_0))^{-1}J_0(ra_0)$ for $0<r<\rho$, and by \cite[(10.22.7)]{dlmf}, $\alpha_0=\rho^2/2.$

The next lemma extends Lemma \ref{l:2c}.  In addition to allowing more general potentials, it shows that on $\Lambda$ there are many resonances near $0$ for both positive and negative values of $\ve$.
\begin{lem} \label{l:evpert}   
Suppose $V_\ve$ satisfies 
Assumption \ref{a:Vve}, $-\Delta +V_0$ has an eigenvalue at $0$, with a corresponding eigenfunction $\psi$, $\partial_\theta \psi =i\ell 
\psi$,  $|\ell| \ge 2$, and $\| \psi\|_{L^2(\Real^2)}=1$.
 Use
the notation 
$$O'_m(\ve,\ell)=\delta_{|\ell|,2}O_m(\ve \log |\ve|)+O_m(|\ve|^{2-\delta})$$ for arbitrary $\delta >0$.
 If $\alpha_0:= \int_{\Real^2}|\psi (r,\theta)|^2 V(r) dx  \not =0$, then there is an $\alpha_1\in \Real$ so that
for each $m\in \Integers$ there is an $\ve_0=\ve_0(m)>0$ such that
\begin{enumerate}
\item  If $0< -\ve \alpha_0 <|\alpha_0|\ve_0$, 
then there is a resonance at $\lambda_\ve(m)$  satisfying $$|\lambda_\ve(m)|= \sqrt{|\ve \alpha_0+\alpha_1\ve^2|}+|\ve|^{1/2}O_m'(\ve,\ell)\;\text{ and} \;
\arg \lambda_\ve(m)= \pi /2 +m \pi +O'_m(\ve,\ell).$$
\item  If $0< \ve \alpha_0< |\alpha_0|\ve_0$, then there is a resonance at $\lambda_\ve(m) $ satisfying 
$$|\lambda_\ve(m)|=\sqrt{\ve \alpha_0+\ve^2\alpha_1}+|\ve|^{1/2}O_m'(\ve,\ell)\; \text{and }\; \arg \lambda_\ve(m) = m \pi +O_m'(\ve,\ell).$$
\end{enumerate}
When $m=0$ and $0<-\alpha_0\ve<|\alpha_0|\ve_0$, the resonance corresponds to a negative eigenvalue 
of $-\Delta +V_\ve$ at $\ve \alpha_0+ \ve^2\alpha_1+O(\ve^{2}(|\ve|^{1-\delta}+\delta_{|\ell|,2}\log |\ve|))$.
In each case the resonance or eigenvalue corresponds to a pole of $\Rve \proj_\ell$.
\end{lem}
By including the $\alpha_1$ term we see that the resonances lie quite close to $\arg \lambda =m\pi$ or $\arg \lambda =(m+1/2) \pi $ when $|\ell|>2$; compare the
numerically computed resonances for  $\ell =2$ and $\ell=3$ in Figure \ref{f:3foldres}.
If $|\ell|=2$ we can take $\alpha_1=0$ since the error is larger
than the correction made by including $\alpha_1$.
\begin{proof}
The proof of this lemma begins in a manner similar to that of Lemma \ref{l:srs}.
We first consider the case $| \ell|> 2$.
From the expansions of the resolvent from \cite[Theorem 1]{cdgen} and using that the potential $V_0$ is rotationally symmetric
we see that if $|\ell|>2$,
\begin{equation}\label{eq:rev}R_{V_0}(\lambda) \proj_\ell = -\frac{1}{\lambda^2} \psi\otimes \psi+ B_{0,0}  \proj_\ell+ O(|\lambda|^{2-\delta})
\end{equation}
where $B_{0,0}: L^2_c(\Real^2)\rightarrow H^2_{\loc }(\Real^2)$ and the error is as an operator $L^2_c(\Real^2)\rightarrow H^2_{\loc }(\Real^2)$.  The error is uniform in a finite range of $\arg \lambda$.   
By the Birman-Schwinger principle
$\Rve$ has a pole at $\lambda$, a regular point of $\Rvz$, if and only if $-1$ is an eigenvalue of $(V_\ve-V_0)R_{V_0} (\lambda)\proj_\ell \chi$.
 
Set $\tilde{R}_{V_0}(\lambda)= R_{V_0}(\lambda)+ \lambda^{-2}\psi\otimes \psi$.  For $\arg \lambda$ in a bounded set and $|\lambda|, \; |\ve|$ sufficiently small, 
$\tilde{R}_{V_0}(\lambda)\proj_l$ is analytic away from $\lambda=0$, bounded at $\lambda=0$, and $I+(V_\ve-V_0) \tilde{R}_{V_0}(\lambda)\proj_l\chi$ is 
invertible. 
Thus, as in the proof of Lemma \ref{l:srs}, we reduce the problem to that of studying the zeros of
\begin{equation}\label{eq:fep}
Q(\lambda, \ve):= \det \left( I-(V_\ve-V_0) \frac{1}{\lambda^2}\psi\otimes \psi (I +(V_\ve-V_0) \tilde{R}_{V_0} \proj_l \chi)^{-1}\right).
\end{equation}
Away from $\lambda=0$ this is an analytic function of $\lambda$ in this region.
Using that we are taking the determinant of the identity plus a rank one operator as in Lemma \ref{l:srs} the expansion \eqref{eq:rev} shows that 
\begin{equation}
\label{eq:fexp}
Q(\lambda,\ve)= 1-\frac{\ve}{\lambda^2} \left( \alpha_0 + \ve \alpha_1+ O(\ve^2)+  O(\ve  |\lambda|^{2-\delta})\right)
\end{equation}
where $\alpha_1=-\langle B_{00}\dot{V}_0 \psi,  \dot{V}_0\psi \rangle +\frac{1}{2}\langle \ddot{V}_0 \psi,\psi\rangle$.  Because $B_{00}\proj_\ell$ is self-adjoint, $\alpha_1$ is real.

Set
$$g(\lambda,\ve)=  1-\frac{\ve}{\lambda^2} \left( \alpha_0 + \ve \alpha_1 \right). $$
Fix $m\in \Integers, $ and let $\tl_\ve(m)\in \Lambda$ be the continuous function of $\ve$ satisfying 
$\tl^2_\ve(m)= \ve \alpha_0+\alpha_1 \ve^2$ (so that 
$g(\tl_\ve,\ve)=0$) and $|\arg \tl_\ve(m)-m\pi- \pi/4|\leq \pi/4$.    
Notice that if $\alpha_0\ve>0$ then $\arg \tl_\ve(m)=m\pi$, and if $\alpha_0\ve<0$ then $\arg \tl_\ve(m)=m\pi+\pi/2$ when $|\ve| $ is sufficiently small.
We identify  $\{\lambda \in \Lambda: |\arg \lambda- \pi(m+1/4)|<\pi/2\}$ with a subset of $\Complex$.  For 
$w\in \Complex,$ $|w|$ small compared to $|\tl_\ve|$,  $|(\tl_\ve+w)^2g(\tl_\ve+w,\ve)|\geq 2 |\tl_\ve w|+O( |w|^2)$.  Using this,
\eqref{eq:fexp}, and Rouch\'{e}'s Theorem
applied to the pair $\lambda^2 Q$ and $\lambda^2 g$
shows that $Q$ has a  simple zero, and hence 
$-\Delta+V_\ve$ has a resonance, at a point $\lambda_{\ve}(m)$ satisfying $|\lambda_\ve (m)-\tl_\ve(m)|=O(\ve)$ when $|\ve|$ is sufficiently small, depending on $m$.
  Using \eqref{eq:fexp} allows us to improve the error to $O_m(|\ve|^{5/2-\delta})$.  
  
  If $0<-\alpha_0 \ve$ is sufficiently small, then $\arg(\lambda_\ve(0))=\pi/2$ and $-\Delta+V_\ve$ has an eigenvalue at
  $(\lambda_\ve(0))^2=\ve(\alpha_0+\alpha_1\ve)+O(|\ve|^{3-\delta}).$
  This completes the proof for $|\ell|>2$.

  If $|\ell |=2$, the expansion \eqref{eq:rev} must be replaced by 
  $$R_{V_0}(\lambda) \proj_\ell = -\frac{1}{\lambda^2} \psi\otimes \psi+ O(\log \lambda). $$
  Again, for $|\lambda|$ and $|\ve|$ sufficiently small the poles of $\Rve$ away from $\lambda=0$ are given
  by the zeros of the function $Q$  defined  in \eqref{eq:fep}, but the expansion
  \eqref{eq:fexp} must be replaced by 
  $$Q(\lambda,\ve)= 1-\frac{\ve}{\lambda^2} \left( \alpha_0 + O(\ve \log \lambda)\right)$$
  and we take $\alpha_1=0$ for simplicity.
  The proof proceeds as in the $|\ell|>2$ case, except that we use 
  $g(\lambda,\ve)= 1-\ve \alpha_0/\lambda^2$, $\tl_\ve(m)^2=\ve \alpha_0$,   and get a  worse error.
\end{proof}
Suppose $\ell\geq 2$ and $V_\ve = (-a_0^2+\ve)\bbo_\rho(r)$, with $J_{\ell-1}(a_0\rho)=0$ so that 
$-\Delta+V_0$ has an eigenvalue at $0$. Then, in the notation of Lemma \ref{l:evpert}, using \cite[10.22.38]{dlmf}
$$\psi_l(r,\theta)=
\begin{cases}
\sqrt{\frac{\ell-1}{\pi \ell}}\frac{1}{\rho J_\ell(a_0 \rho)}J_\ell(a_0r) e^{i\ell \theta},& 0<r<\rho\\
\sqrt{\frac{\ell-1}{\pi \ell}}\frac{1}{\rho}(r/\rho)^{-\ell} e^{i\ell \theta},& r\geq \rho
    \end{cases}
$$
and $\alpha_0= (\ell-1)/\ell.$

The next lemma extends Lemma \ref{l:pwrc}, and shows
that not only does a family of negative eigenvalues in the $\ell=\pm 1$ mode persist, but that it corresponds to many resonances for both
signs of $\ve$.  Since $\log \lambda$ is a natural coordinate
to use on $\Lambda$ and because $\log \lambda$ naturally arises
in the proof of Lemma \ref{l:pwrg}, we use $\log \lambda$ in the statement of the lemma.
\begin{lem}\label{l:pwrg} Let $V_\ve$ satisfy assumption \ref{a:Vve}.
Suppose $-\Delta +V_0$ has a $p$ wave resonance with corresponding resonant states $\psi_{\pm}$
satisfying 
 $\psi_\pm (r,\theta)=e^{\pm i \theta}/(r\sqrt{2\pi})$ for $r$ sufficiently large, and define $\alpha_0:=\int_{\Real^2} \dot V_0(r)|\psi(r,\theta)|^2 dx$ and 
$$b= \log 2-\gamma +i\pi/2 +\lim_{\rho' \rightarrow \infty} \left(\int_{|x|<\rho'}|\psi_{\pm}|^2 dx -\log \rho'\right) .$$  If $\alpha_0\not =0$ then for each $n\in \Integers \setminus \{0\}$ there is an $\ve_0=\ve_0(n)>0$ so that if $|\ve|<\ve_0$, $-\Delta +V_\ve$ has a resonance at a point 
$\lambda_\ve(n)$ satisfying 
$$\log \lambda_\ve(n) = b+\frac{1}{2}W_n(2\ve (\alpha_0+ \alpha_1\ve) e^{-2\re b}) +o(|\ve|).$$ When  $\ve \alpha_0<0$, the $n=-1$ resonance corresponds to $-\Delta+V_\ve$ having a negative eigenvalue given by 
$$ (\lambda_\ve(-1))^2= e^{2b} \exp\left( W_{-1}(2\ve (\alpha_0+\alpha_1 \ve) e^{-2\re b})\right)(1+o(\ve))=-\frac{2\ve(\alpha_0+\alpha_1\ve)(1+o(\ve))}{W_{-1}( 2\ve (\alpha_0+ \alpha_1\ve) e^{-2\re b})}. $$
In each case, $R_{V_\ve}(\lambda)\proj_{\pm 1}$ has a pole
at $\lambda=\lambda_\ve(n)$.
\end{lem}
We note that, by Lemma \ref{l:lwf}, $\lim_{\alpha_0\ve\uparrow 0} \im \log \lambda_\ve(n)= (\pi/2)(2+2n-\sgn n)$ and 
$\lim_{\alpha_0\ve\downarrow 0} \im \log \lambda_\ve(n)= (\pi/2)(1+2n-\sgn n)$.
\begin{proof}
We give the proof for the $\ell =1$ mode; the 
case $\ell =-1$ follows.  We set $\psi =\psi_+$.

From \cite[Theorem 1]{cdgen},  there is 
an  operator $B_{00}:L^2_c\rightarrow H^2_{\loc}$ such that near $\lambda=0$
\begin{equation}\label{eq:artie}
\chi R_{V_{0}}(\lambda)\proj_{1} \chi = \chi \left( \frac{1}{(\log \lambda -b)\lambda^2}  \psi \otimes  \psi  + B_{00}\proj_{1} \right)\chi + O((\log \lambda)^{-1}).
\end{equation}

We set $\tilde{R}_{V_0}(\lambda)= R_{V_0}(\lambda)- \frac{1}{\log \lambda -b}  \psi \otimes  \psi $.
As in the proofs of Lemmas \ref{l:evpert} and \ref{l:srs}, to study the resonances with $|\arg \lambda|<\varphi$ and $0<|\lambda|$, 
for $|\lambda|$ , $|\varepsilon|$ sufficiently small, it suffices to study the zeros of 
\begin{equation}
Q(\lambda,\ve):= \det\left(I+\frac{1}{(\log \lambda-b)\lambda^2}(V_\ve-V_0) \psi \otimes \psi (I+(V_\ve-V_0) \tilde{R}_{V_0}(\lambda)\proj_1\chi)^{-1}\right).
\end{equation}
Using \eqref{eq:artie} yields 
\begin{equation}\label{eq:Qprs}
Q(\lambda,\varepsilon)= 1 + \frac{\varepsilon}{(\log \lambda -b)\lambda^2}\left( \alpha_0 +\alpha_1\ve +O(\varepsilon (\log \lambda)^{-1})+O(\ve^2)\right)
\end{equation}
where $\alpha_1=- \langle B_{00}\dot{V}_0\psi,\dot{V}_0\psi \rangle+\frac{1}{2}\langle \ddot{V}_0\psi,\psi\rangle\in \Real$.
Now set
\begin{equation}\label{eq:gprs}
g(\lambda,\ve)=  1 + \frac{\ve}{(\log \lambda -b)\lambda^2}(\alpha_0+ \alpha_1 \ve) .
\end{equation}
Solving $g(\lambda,\ve)=0$ as in Lemma \ref{l:pwrc} leads for $n\in \Integers\setminus \{ 0\}$ to solutions $\tl_\ve(n)$ satisfying 
 $\log \tl_\ve(n) =  b+\frac{1}{2}W_n(2\epsilon(\alpha_0+ \alpha_1 \epsilon) e^{-2\re b})$.

We apply Rouch\'{e}'s Theorem to the pair $g$, $Q$ for the curve $|\tl_\ve(n)-\lambda|=|\ve|^{1/4}|\tl_\ve(n)|$.  This shows that for each $n\in \Integers \setminus \{0\}$ there
is an $\ve_0(n)$ so that there is a zero $\lambda_\ve(n)$ of $Q(\lambda,\ve)$ within $|\ve|^{1/4}|\tl_\ve(n)|$ of $\tl_\ve(n)$ when $|\ve|<\ve_0(n)$.  Using \eqref{eq:Qprs} and \eqref{eq:gprs} shows 
that we can improve the error to 
$\log \lambda_\ve(n) -\log \tl_\ve(n)=o(\ve)$, or 
$\lambda_\ve(n)-\tl_\ve(n)= o(|\ve|^{3/2}).$

  If $n=-1$ and $\epsilon \alpha_0<0$, then Lemma \ref{l:lwf} shows that the resulting pole
of $R_\epsilon \proj_1$ corresponds to an eigenvalue, since the pole lies in the physical space.
\end{proof}

Above  we have restricted our considerations to the 
radially symmetric case in an effort to make the paper
as readable as possible.  However, radial symmetry really is not essential to the proofs.
What is essential is 
the ability to reduce the problem to the study of a family of 
operators of $\lambda $ and $\ve$ with a singularity of 
rank $1$ at $\lambda =0$ when $\ve=0$.  For the radial
case we did this by projecting $I+(V_\ve-V_0)R_{V_0}(\lambda)$ onto a Fourier mode.
If the operator $R_{V_0}(\lambda)$ has a singularity of rank $1$ 
at $\lambda =0$, there is no need for a projection, and so no need
for an assumption of radial symmetry.
As an 
example, we include the following lemma which
is an analog of Lemma \ref{l:evpert} without the assumption of radial symmetry.  Analogs of Lemmas \ref{l:srs} and \ref{l:pwrg} are equally possible, though in general 
one should expect the errors to be worse than in those 
lemmas.
\begin{lem}Suppose $V_\ve$ is a family of compactly
supported real-valued functions in $L^\infty(\Real^2)$ which depend in a $C^\infty$ way on $\ve$.  Suppose $0$ is 
a simple eigenvalue of $-\Delta +V_0$ and $-\Delta+V_0$ has neither
an s-wave resonance nor a p-wave resonance, and suppose
$\psi\in L^2(\Real^2)$ satisfies $(-\Delta +V_0)\psi=0$, $\|\psi\|=1$.  
If $\alpha_0:= \int_{\Real^2}|\psi (x)|^2 \dot{V}_0(x) dx  \not =0$, then 
for each $m\in \Integers$ there is an $\ve_0=\ve_0(m)>0$ so that\\
(1)  if $0< -\ve \alpha_0 <|\alpha_0|\ve_0$, 
then there is a resonance at $\lambda_\ve(m)$,  satisfying 
$$|\lambda_\ve(m)|= \sqrt{|\ve \alpha_0|}+O_m(|\ve|^{3/2}\log |\ve|)) \; \text{and}\; 
\arg \lambda_\ve(m)= \pi /2 +m \pi +O_m(|\ve|\log |\ve|).$$\\
(2) if $0< \ve c_0< |c_0|\ve_0$ there is a resonance at $\lambda_\ve(m) $ satisfying 
$$|\lambda_\ve(m)|=\sqrt{\ve \alpha_0}+O_m(|\ve|^{3/2}\log |\ve|)\; \text{and}\; \arg \lambda_\ve(m) = m \pi +O_m(|\ve|\log |\ve|)).$$
When $m=0$ and $0<-\alpha_0\ve<|\alpha_0|\ve_0$, the resonance corresponds to a negative eigenvalue 
of $-\Delta +V_\ve$ at $\ve \alpha_0+ O(\ve^2 \log|\ve|)$.
\end{lem}

\section{The scattering matrix and scattering phase for the circular well}
In this section we return to the circular well potential, $V= -a^2 \bbo_\rho$, with $a,\; \rho>0$ fixed, and study the derivative of the scattering phase, $\sigma'(\lambda)$, as $
\lambda \downarrow 0$, proving Theorem \ref{t:spintro}.
Because the potential is radial, the scattering matrix $S(\lambda)$ is diagonal when using the eigenfunctions $e^{i\ell \theta}$ as a basis.
The eigenvalue $S_\ell$ of $S(\lambda)$ associated with $e^{i\ell\theta}$ is the ratio of the outgoing amplitude to the incoming amplitude for elements of the null space of $P-\lambda^2$:
\begin{equation}\label{eq:eigenS}
    S_\ell = \frac{c_\ell}{d_\ell} = -\frac{B^{(2)}_\ell}{B^{(1)}_\ell},\quad B^{(k)}_\ell(\lambda,\rho,a) := H^{(k)}_\ell(\lambda\rho)+A_\ell H^{(k)\prime}_\ell(\lambda\rho),\quad A_\ell(\lambda,\rho,a):= -\frac{\lambda}{\mu}\frac{J_\ell(\mu\rho)}{J_\ell'(\mu\rho)}.
\end{equation}
We note that $S_{-\ell}=S_{\ell}$ for all $(\lambda,\rho,a,\ell)$ since $\mathcal{C}_{-\ell}(z)=(-1)^{\ell}\mathcal{C}_{\ell}(z)$.
Thus \eqref{e:sigmadef} becomes
\begin{equation}\label{eq:sPhase}
    \sigma(\lambda) = \frac{1}{2\pi i}\sum_{\ell=-\infty}^{\infty}\log S_\ell = \sigma_0 + 2\sum_{\ell=1}^\infty \sigma_\ell,
\end{equation}
where the $\sigma_\ell$ are given by $\sigma_\ell(0) = 0$ and the following lemma, which among other things proves that the sum in \eqref{eq:sPhase} converges very rapidly.
\begin{lem}\label{l:spfundamentals}
    For every $\ell \in \mathbb N_0$, $\lambda\in \Lambda$, $a>0$ and $\rho>0$, we have
    \begin{equation}\label{eq:siglprime}
        \sigma'_\ell = -\frac{a^2}{\mu^2}\frac{2}{\pi^2\lambda}\frac{1-\frac{\ell^2}{(\lambda\rho)^2}A_\ell^2}{(J_\ell(\lambda\rho)+A_\ell J'_\ell(\lambda\rho))^2+(Y_\ell(\lambda\rho)+A_\ell Y'_\ell(\lambda\rho))^2}.
    \end{equation}
    Moreover, given $\lambda_0>0$, $a_0>0$, and $\rho_0>0$, there is $\ell_0 >0$ such that 
    \begin{equation}\label{eq:siglprimebd}
    |\sigma_\ell'(\lambda) | \lesssim \frac {\ell^3} {\mu^2 \lambda  } \Big(\frac {\lambda \rho e} {2\ell} \Big)^{2\ell},
    \end{equation}
    uniformly for $\ell \ge \ell_0$, $\lambda \in (0,\lambda_0]$, $a \in (0,a_0]$, and $\rho \in (0,\rho_0]$.
\end{lem}
\begin{proof} We have
\begin{equation}\label{eq:spfirst}
   2\pi i  \sigma'_\ell(\lambda) = \frac{\partial_\lambda B^{(2)}_\ell}{ B^{(2)}_\ell}-\frac{\partial_\lambda B^{(1)}_\ell}{B^{(1)}_\ell} = \frac{B^{(1)}_\ell\partial_\lambda B^{(2)}_\ell-B^{(2)}_\ell\partial_\lambda B^{(1)}_\ell}{B^{(1)}_\ell B^{(2)}_\ell}.
\end{equation}
    To simplify the denominator, we expand by writing the Hankel functions in terms of $J_\ell$ and $Y_\ell$:
    \begin{equation*}
    \begin{aligned}
        B^{(1)}_\ell B^{(2)}_\ell 
        &=(J_\ell(\lambda\rho)+A_\ell J'_\ell(\lambda\rho))^2+(Y_\ell(\lambda\rho)+A_\ell Y'_\ell(\lambda\rho))^2.
    \end{aligned}
    \end{equation*}
    As for the numerator we first use Bessel's equation on $A'_\ell$:
    \begin{equation*}
    \begin{aligned}
        A'_\ell&= \frac{-a^2}{\mu^3}\frac{J_\ell(\mu\rho)}{J_\ell'(\mu\rho)}-\frac{\lambda^2\rho}{\mu^2}+\frac{\lambda^2\rho}{\mu^2}\frac{J_\ell(\mu\rho)J_\ell''(\mu\rho)}{(J_\ell'(\mu\rho))^2}\\
        &=\frac{-1}{\mu}\frac{J_\ell(\mu\rho)}{J_\ell'(\mu\rho)} - \frac{\lambda^2\rho}{\mu^2}\left(1+\left(\frac{J_\ell(\mu\rho)}{J_\ell'(\mu\rho)}\right)^2\left(1-\frac{\ell^2}{(\mu\rho)^2}\right)\right).
    \end{aligned}
    \end{equation*}
    Using Bessel's equation on the Hankel functions, we combine in a similar manner for $\partial_\lambda B^{(2)}_\ell$ to get
    \begin{equation*}
        \partial_\lambda B_\ell^{(2)}
        =-\rho A_\ell\left(1 -\frac{\ell^2}{(\lambda\rho)^2}\right)H^{(2)}_\ell(\lambda\rho)- \rho A_\ell^2\left(1-\frac{\ell^2}{(\mu\rho)^2}\right)H^{(2)\prime}_\ell(\lambda\rho) + \frac{a^2\rho}{\mu^2}H^{(2)\prime}_\ell(\lambda\rho).
    \end{equation*}
    Multiplying by $B^{(1)}_\ell$ gives us the expanded form
    \begin{equation*}
    \begin{aligned}
        B_\ell^{(1)}\partial_\lambda B_\ell^{(2)} =&\ -\rho A_\ell\left(1 -\frac{\ell^2}{(\lambda\rho)^2}\right)H^{(2)}_\ell(\lambda\rho)H^{(1)}_\ell(\lambda\rho)- \rho A_\ell^2\left(1-\frac{\ell^2}{(\mu\rho)^2}\right)H^{(1)}_\ell(\lambda\rho)H^{(2)\prime}_\ell(\lambda\rho)\\
        &\ + \frac{a^2\rho}{\mu^2}H^{(1)}_\ell(\lambda\rho)H^{(2)\prime}_\ell(\lambda\rho)-\rho A_\ell^2\left(1 -\frac{\ell^2}{(\lambda\rho)^2}\right)H^{(1)\prime}_\ell(\lambda\rho)H^{(2)}_\ell(\lambda\rho)\\
        &\ - \rho A_\ell^3\left(1-\frac{\ell^2}{(\mu\rho)^2}\right)H^{(1)\prime}_\ell(\lambda\rho)H^{(2)\prime}_\ell(\lambda\rho) + A_\ell\frac{a^2\rho}{\mu^2}H^{(1)\prime}_\ell(\lambda\rho)H^{(2)\prime}_\ell(\lambda\rho).
    \end{aligned}
    \end{equation*}
    The product $B^{(2)}_\ell \partial_\lambda B_\ell^{(1)}$ is given by the same form as above but swapping $(1)\leftrightarrow(2)$ in the superscripts of the Hankel functions, so we get
    \begin{equation*}
    \begin{aligned}
        B_\ell^{(1)}&\partial_\lambda B_\ell^{(2)}-B_\ell^{(2)}\partial_\lambda B_\ell^{(1)} \\
        =&\ \left[- \rho A_\ell^2\left(1-\frac{\ell^2}{(\mu\rho)^2}\right)+\frac{a^2\rho}{\mu^2}+\rho A_\ell^2\left(1 -\frac{\ell^2}{(\lambda\rho)^2}\right)\right]\left[H^{(1)}_\ell(\lambda\rho)H^{(2)\prime}_\ell(\lambda\rho)-H^{(1)\prime}_\ell(\lambda\rho)H^{(2)}_\ell(\lambda\rho)\right]\\
        =&\ 2i\left[- \rho A_\ell^2\left(1-\frac{\ell^2}{(\mu\rho)^2}\right)+\frac{a^2\rho}{\mu^2}+\rho A_\ell^2\left(1 -\frac{\ell^2}{(\lambda\rho)^2}\right)\right]\left[J'_\ell(\lambda\rho)Y_\ell(\lambda\rho)-J_\ell(\lambda\rho)Y'_\ell(\lambda\rho)\right].
    \end{aligned}
    \end{equation*}
    Using Wronskian identities on the remaining Bessel functions we get
    \begin{equation*}
    \begin{aligned}
    B_\ell^{(1)}\partial_\lambda B_\ell^{(2)}-B_\ell^{(2)}\partial_\lambda B_\ell^{(1)}  &= 
        -\frac{4i}{\pi\lambda\rho}\left[- \rho A_\ell^2\left(1-\frac{\ell^2}{(\mu\rho)^2}\right)+\frac{a^2\rho}{\mu^2}+\rho A_\ell^2\left(1 -\frac{\ell^2}{(\lambda\rho)^2}\right)\right]\\
        &=\frac{4i}{\pi\lambda}\left[ \frac{(a\ell)^2}{(\lambda\mu\rho)^2}A_\ell^2-\frac{a^2}{\mu^2}\right]
        =-\frac{a^2}{\mu^2}\frac{4i}{\pi\lambda}\left( 1-\frac{\ell^2}{(\lambda\rho)^2}A_\ell^2\right).
    \end{aligned}
    \end{equation*}
    Combining the two expressions for $B_\ell^{(1)}\partial_\lambda B_\ell^{(2)}-B_\ell^{(2)}\partial_\lambda B_\ell^{(1)}$ and $B^{(1)}_\ell B^{(2)}_\ell$ 
    with \eqref{eq:spfirst} yields \eqref{eq:siglprime}.

To prove \eqref{eq:siglprimebd}, we use the notation $f \asymp g$ to mean $f \lesssim g$ and $g \lesssim f$. By Olver's uniform large order Bessel function asymptotics (see e.g. \cite[(A.7)]{dgs}), for every $x_0 >0$ there is $\ell_0>0$ such that
\[
J_\ell(x) \asymp  x J'_\ell(x), \qquad - Y_\ell(x) \asymp  x Y'_\ell(x) \asymp \ell^{-1/2} \exp(\ell \xi),
\]
\[
\xi = \int_{x/\ell}^1 \sqrt{t^{-2}-1}dt = \log(\ell/x) + \log 2
 - 1 + O(x^2/\ell^2),\]
uniformly for $x \in (0,x_0]$ and $\ell \ge \ell_0$. Hence 
\[
-A_\ell \asymp \lambda \rho,
\]
and
\[
|\sigma'(\lambda)| \lesssim \frac {\ell^2} {\mu^2 \lambda  Y_\ell(\lambda \rho)^2} \lesssim \frac {\ell^3} {\mu^2 \lambda  } \Big(\frac {\lambda \rho e} {2\ell} \Big)^{2\ell}.
\]
This proves \eqref{eq:siglprimebd}. 
Note that for $\lambda,\rho,a$ fixed we can get  \eqref{eq:siglprimebd} from the simpler large order, fixed argument, asymptotics of Section 10.19(i) of \cite{dlmf}.
\end{proof}

We remark that \eqref{eq:siglprimebd} implies that
$$\sigma'(\lambda)= \sigma_0'(\lambda)+2\sum_{\ell=1}^\infty \sigma_\ell'(\lambda)= \sigma_0'(\lambda)+2\sum_{\ell=1}^{\ell_0} \sigma_\ell'(\lambda)+O(\lambda^{2\ell_0+1})\; \text{as}\; \lambda \rightarrow 0$$
when $\ell_0$ is sufficiently large.

In this section we prove an expansion for $\sigma'(\lambda)$ for small $\lambda$. 
We will use repeatedly the general property that 
\begin{equation} \label{eq:geos}
\text{if } r(x) = o(f(x)) \text{ then } \frac{1}{f(x)+r(x)}=\frac{1}{f(x)}-\frac{r(x)}{(f(x))^2}+ O\left( \left( \frac{r(x)}{(f(x))^2}\right)^2\right).
\end{equation}

  We begin with $\sigma_0'$.
  \begin{lem}\label{l:s'0}
For $\lambda\to 0$, we have the following asymptotic expressions for $\sigma'_0(\lambda)$:
\begin{equation}
    \sigma'_0 (\lambda)= \begin{dcases}
        -\frac{\rho^4}{8}\lambda^3 + O(\lambda^5\log\lambda),& J_1(a\rho)= 0\\
        \frac{-2/\lambda}{4(\log(\lambda/2) + C(\rho,a) + \gamma)^2+\pi^2} + O(\lambda/(\log\lambda)^2),& J_1(a\rho)\neq 0,
    \end{dcases}
\end{equation}
where 
\begin{equation*}
    C(\rho,a) = \log\rho + \frac{1}{a\rho}\frac{J_0(a\rho)}{J_1(a\rho)}.
\end{equation*}
\end{lem}
Compare \cite[(1.11)]{cdobs}, for asymptotics of the derivative of the scattering phase in the Dirichlet obstacle case.
\begin{proof}
First we assume $J_1(a\rho)\neq 0$.  Using 
\begin{equation}\label{eq:simplebf0}
J_0'(z)=-J_1(z),\; Y_0'(z)=-Y_1(z)
\end{equation}
and  expanding the two terms in the denominator of \eqref{eq:siglprime} using \eqref{eq:Yexp} yields
\begin{equation*}
\begin{aligned}
    \left(J_0(\lambda\rho) - A_0 J_1(\lambda\rho)\right)^2 &= 1 + O(\lambda^2),\\
    \left(Y_0(\lambda\rho) - A_0 Y_1(\lambda\rho)\right)^2 & = \left[ \frac{2}{\pi}(\log (\lambda \rho/2) +\gamma)+\frac{2}{\pi a \rho}\frac{J_0(a \rho)}{J_1(a\rho)} + O(\lambda^2 \log \lambda)\right]^2 \\ 
    &=\frac{4}{\pi^2}\left(\log(\lambda/2)+\gamma+C(\rho,a)\right)^2+O(\lambda^2\log^2\lambda).
\end{aligned}
\end{equation*}
Putting this into \eqref{eq:siglprime} we get
\begin{equation}\label{eq:sig0int}
    \sigma'_0 (\lambda)= -\frac{a^2}{\mu^2}\frac{2}{\pi^2\lambda}\frac{1}{1 + \frac{4}{\pi^2}\left(\log(\lambda/2)+\gamma+C(\rho,a)\right)^2+O(\lambda^2\log^2\lambda)}.
\end{equation}
Using \eqref{eq:geos}) gives
$$\mu^{-2}(\lambda)= \frac{1}{a^2} -\frac{\lambda^2}{a^4}+O(\lambda^4).$$
Applying these observations to \eqref{eq:sig0int} yields
\begin{align*}
  \sigma'_0 (\lambda) &= -\frac{2}{\pi^2\lambda}\left[\frac{1}{1 + \frac{4}{\pi^2}\left(\log(\lambda/2)+\gamma+C(\rho,a)\right)^2} + O(\lambda^2/(\log\lambda)^2)\right]\\
    &=\frac{-2/\lambda}{4\left(\log(\lambda/2)+\gamma+C(\rho,V)\right)^2+\pi^2} + O(\lambda/(\log\lambda)^2).
\end{align*}

Now consider the case $J_1(\rho a)=0$.  Here we  again use \eqref{eq:simplebf0}, and multiply  both the numerator and 
denominator of \eqref{eq:siglprime} by $\mu^2 (J_1(\mu \rho))^2$ to get
\begin{equation}\label{eq:so'alt}
\sigma_0'(\lambda)=-a^2 \frac{2}{\pi^2\lambda} \frac{ (J_1(\mu \rho))^2}{ (\mu J_1(\mu \rho) J_0(\lambda\rho)-\lambda J_0(\mu \rho)  J_1(\lambda\rho))^2+(\mu J_1(\mu \rho) Y_0(\lambda\rho)-\lambda J_0(\mu \rho)Y_1(\lambda\rho))^2}.
\end{equation}
Since in this case 
\begin{equation}\label{eq:J1exp}
J_1(\mu \rho)=  J_1'(a \rho)\frac{\rho \lambda^2}{2a}+O(\lambda^4)= -J_0(a \rho)\frac{\rho \lambda^2}{2a}+O(\lambda^4)
\end{equation}
we get from the expansions \eqref{eq:Yexp}
\begin{align}\label{eq:0de}
(\mu J_1(\mu \rho) J_0(\lambda\rho)-\lambda J_0(\mu \rho)  J_1(\lambda\rho))^2& = O(\lambda^4)\nonumber \\
(\mu J_1(\mu \rho) Y_0(\lambda\rho)-\lambda J_0(\mu \rho)Y_1(\lambda\rho))^2&= \left( \frac{2 J_0(\rho a)}{\pi  \rho} +O(\lambda^2 \log \lambda)\right)^2.
\end{align}
Using \eqref{eq:J1exp} and \eqref{eq:0de} in \eqref{eq:so'alt} gives the result when $J_1(a\rho)=0$.
\end{proof}

We give an alternate expression for $\sigma_\ell'$.   
Using the recurrence relations \eqref{eq:recur}, we find 
\begin{align}\label{eq:intermediate}
\mu^2 (J_{\ell}'(\mu \rho))^2-\frac{\ell^2}{\rho^2}J_{\ell}(\mu \rho))^2& = \mu^2 \left[ \Big(J_{\ell-1}(\mu \rho)-\frac{\ell}{\mu \rho}J_{\ell}(\mu\rho)\Big)^2-\frac{\ell^2}{(\mu \rho)^2}(J_\ell(\mu \rho))^2\right]\nonumber \\
& = -\mu^2 J_{\ell-1}(\mu \rho)J_{\ell +1}(\mu \rho).
\end{align}
Multiplying the numerator and denominator in \eqref{eq:siglprime} by $(\mu J_\ell'(\mu \rho))^2$
and using the identities \eqref{eq:recur} and \eqref{eq:intermediate}, yields that for $\ell \geq 1$
\begin{equation}\label{eq:sigell'alt}
\sigma_\ell'(\lambda) = \frac{2 a^2}{\pi^2 \lambda} \frac{ J_{\ell-1}(\mu \rho) J_{\ell +1} (\mu \rho)}
{( \mu J_{\ell-1}(\mu \rho)J_{\ell}(\lambda \rho)-\lambda J_\ell (\mu \rho)J_{\ell -1} (\lambda \rho))^2 + 
( \mu J_{\ell-1}(\mu \rho)Y_{\ell}(\lambda \rho)-\lambda J_\ell (\mu \rho)Y_{\ell -1} (\lambda \rho))^2}.
\end{equation}
This expression shows perhaps more explicitly than  \eqref{eq:siglprime} why we should expect
different sorts of behavior in the three cases $J_{\ell+ 1}(a\rho)=0$, $J_{\ell-1}(\rho a)=0$,
 and $J_{\ell -1}(\rho a)J_{\ell+1}(\rho a)\not =0$.  

The next lemma will be used to find
 asymptotics of $\sigma_{\ell}'(\lambda)$ for $\ell\geq 1$ when $J_{\ell -1}(\rho a)\not =0$.   
 \begin{lem}\label{l:intlgeq1}
If $\ell \geq 1$ then if $J_{\ell-1}(\rho a)\not =0$, then
as $\lambda\to 0$
\begin{equation}
    \sigma'_\ell (\lambda) = \left( 
   \frac{ \rho }{[(\ell -1)!]^2 J_{\ell -1}(\rho a)}  \left(\frac{\rho \lambda}{2}\right)^{2\ell -1} + O(\lambda^{2\ell +1}) +\delta_{\ell,1}O(\lambda^3 \log \lambda)
   \right) J_{\ell+1}(\rho \mu) .
    \end{equation}
\end{lem}
\begin{proof}
%
The assumptions $J_{\ell-1}(\rho a)\not =0$ and $\ell \geq 1$ ensure that
the leading term in the denominator of \eqref{eq:sigell'alt} as $\lambda \rightarrow 0$
is determined by $\mu J_{\ell-1}(\mu \rho)Y_\ell(\lambda \rho)$.  We have
\begin{equation} 
 \mu J_{\ell-1}(\mu \rho)Y_{\ell}(\lambda \rho)-\lambda J_\ell (\mu \rho)Y_{\ell -1} (\lambda \rho) = a J_{\ell -1}(a\rho) \left( -\frac{(\ell -1)! }{\pi} \right)\left( \frac{2}{\rho \lambda}\right)^\ell+O(\lambda^{-\ell +2})+\delta_{1,\ell}O(\lambda \log \lambda).
\end{equation}
The asymptotics of $J_\ell$, $J_{\ell-1}$ at $0$ yield $\mu J_{\ell-1}(\mu \rho)J_{\ell}(\lambda \rho)-\lambda J_\ell (\mu \rho)J_{\ell -1} (\lambda \rho) =O(\lambda^\ell)$.
Using these gives
\begin{align*}
& \left( ( \mu J_{\ell-1}(\mu \rho)J_{\ell}(\lambda \rho)-\lambda J_\ell (\mu \rho)J_{\ell -1} (\lambda \rho))^2 + 
( \mu J_{\ell-1}(\mu \rho)Y_{\ell}(\lambda \rho)-\lambda J_\ell (\mu \rho)Y_{\ell -1} (\lambda \rho))^2\right)^{-1} \\ & 
= \frac{\pi ^2 }{[a J_{\ell -1}(\rho a) (\ell -1)!]^2 }\left(\frac{\rho \lambda}{2}\right)^{2\ell}+ O( \lambda^{2\ell +2})+ \delta_{1,\ell}O(\lambda^4 \log \lambda)
\end{align*}
Putting this in \eqref{eq:sigell'alt} proves the lemma.
\end{proof}

Next we prove an expansion for $A_\ell$ which will be helpful in the $\ell\geq 1$, $J_{\ell -1}(\rho a)=0$ case.
\begin{lem}\label{lem:Al}
For $\ell\geq 1$, if $J_{\ell-1 }(a\rho)=0$ then
\begin{equation}
A_\ell (\lambda) =
        \frac{\lambda\rho}{\ell}-\frac{(\lambda\rho)^3}{2\ell^2}+O(\lambda^5)
        \end{equation}
\end{lem}
\begin{proof}
Using the recurrence relations
\eqref{eq:recur} we obtain
\begin{equation*}
    A_\ell (\lambda)
    =\frac{-\lambda\rho}{\ell}\frac{J_{\ell-1}(\mu\rho)+J_{\ell+1}(\mu\rho)}{J_{\ell-1}(\mu\rho)-J_{\ell+1}(\mu\rho)}= \frac{-\lambda\rho}{\ell}\left( -1+\frac{2J_{\ell-1}(\mu\rho)}{J_{\ell-1}(\mu\rho)-J_{\ell+1}(\mu\rho)}\right).
\end{equation*}
We will use, again from \eqref{eq:recur},
$$J_{\ell -1}(a \rho)=0\; \Rightarrow \; J_{\ell-1}'(\rho a) = -J_\ell(\rho a)=-\frac{\rho a}{2\ell}J_{\ell +1}(\rho a).$$
Then since for smooth $f$, $f(\mu \rho)= f(a\rho )+f'(a\rho)\rho(\mu -a)+O((\mu-a)^2)$ we get
\begin{align*}
A_\ell (\lambda) &= \frac{-\lambda\rho}{\ell}\left( -1+\frac{2J'_{\ell-1}(a\rho)}{J_{\ell-1}(a\rho)-J_{\ell+1}(a\rho)} \rho (\mu -a) +O( (\mu -a)^2)\right)\\ & 
= \frac{-\lambda\rho}{\ell}\left( -1 +\frac{(\rho \lambda)^2}{2\ell }+O(\lambda^4)\right).
\end{align*}
\end{proof}

\begin{lem}\label{l:sigma1'}
For $\lambda\to 0$, we have the following asymptotic expressions for $\sigma'_1$:
\begin{equation}
    \sigma'_1 (\lambda)= \begin{dcases}
        \frac{-2/\lambda}{4\left(\log(\rho\lambda/2)-1/2+\gamma\right)^2+\pi^2} + O(\lambda/(\log\lambda)^2),& J_0(a\rho)=0\\
        \frac{\rho^4}{8}\lambda^3 + O(\lambda^5\log \lambda),& J_2(a\rho)=0\\
        \frac{J_2(a\rho)}{J_0(a\rho)}\frac{\rho^2}{2}\lambda + O(\lambda^3\log\lambda),& \mathrm{otherwise}.
    \end{dcases}
\end{equation}
\end{lem}
\begin{proof}
First we consider the case $J_0(\rho a)=0$.  For this, we use for $\ell \geq 1$
\begin{equation}\label{eq:cells}
\mathcal{C}_\ell(\lambda \rho )+\frac{\lambda \rho}{\ell}\mathcal{C}_{\ell}'(\lambda \rho)= \frac{\lambda \rho}{\ell}\mathcal{C}_{\ell-1}(\lambda \rho)
\end{equation}
where $\mathcal{C}_\ell$ is $J_\ell$ or $Y_\ell$.  Using this and  Lemma \ref{lem:Al}, 
\begin{align*}Y_1(\lambda\rho)+A_1 Y'_1(\lambda\rho)& = \lambda \rho Y_0(\lambda \rho) -\frac{(\lambda \rho)^3}{2} Y_1'(\lambda \rho)+O(\lambda^3)\\
& = \frac{2}{\pi}\lambda \rho [ \log (\lambda \rho/2)+\gamma]  -\frac{\lambda \rho}{\pi}    + O(\lambda^3 \log \lambda).
\end{align*}
Similarly,
\begin{equation}
J_1(\lambda\rho)+A_1 J'_1(\lambda\rho)  = \lambda \rho J_0(\lambda \rho) -\frac{(\lambda \rho)^3}{2} J_1'(\lambda \rho)+O(\lambda^5)
= \lambda \rho +O(\lambda^3).
\end{equation}
Using these and Lemma \eqref{lem:Al} in \eqref{eq:siglprime} completes the proof for this case.

For the remaining cases we use Lemma \ref{l:intlgeq1}.  If $J_{0}(\rho a)J_{2}(\rho a)\not =0$, we get the third case immediately from Lemma \ref{l:intlgeq1}.  
To handle the remaining case, the second one, note that if $\ell \geq 1$ and $J_{\ell+1}(\rho a)=0$, then  $J_{\ell +1}'(\rho a)=J_\ell(\rho a)$ and 
\begin{equation}\label{eq:Jellp1}
J_{\ell+1}(\rho \mu)= J_{\ell}(\rho a)\frac{\rho \lambda^2}{2a} +O(\lambda^4)\; \text{when}\; J_{\ell+1}(\rho a)=0.
\end{equation}
Moreover, since $J_2(\rho a)=0$, $J_0(\rho a) =\frac{2}{a \rho}J_1(\rho a)$.
Using these in Lemma \ref{l:intlgeq1} with $\ell =1$ completes the proof of the second case.
\end{proof}

The proof of the asymptotic expansion of $\sigma_\ell'$ for $\ell \geq 2$ is similar to the $\ell=1$ case handled by Lemma \ref{l:sigma1'}.
\begin{lem}\label{l:sigmaell'}
For $\lambda\to 0$, we have the following asymptotic expressions for $\sigma'_\ell$ for $\ell\geq 2$:
\begin{equation}
    \sigma'_\ell(\lambda) = \begin{dcases}
        -\frac{\rho}{\ell[(\ell-2)!]^2}\left(\frac{\lambda\rho}{2}\right)^{2\ell-3}+\widetilde O_\ell,& J_{\ell-1}(a\rho)=0\\
        \frac{\rho}{\ell[(\ell-1)!]^2}\left(\frac{\lambda\rho}{2}\right)^{2\ell+1}+O(\lambda^{2\ell+3}),& J_{\ell+1}(a\rho)=0\\
        \frac{J_{\ell+1}(a\rho)}{J_{\ell-1}(a\rho)}\frac{\rho}{[(\ell-1)!]^2}\left(\frac{\lambda\rho}{2}\right)^{2\ell-1}+O(\lambda^{2\ell+1}),& \mathrm{otherwise}.
    \end{dcases}
\end{equation}
where $\widetilde O_2 = O(\lambda^3\log\lambda)$ and $\widetilde O_\ell=O(\lambda^{2\ell-1})$ for $\ell\geq 3$.
\end{lem}
\begin{proof}
We begin with the case $J_{\ell -1}(\rho a)=0$.   Using \eqref{eq:cells} and Lemma \ref{lem:Al} gives
\begin{align}\label{eq:Yells}
Y_\ell(\lambda \rho)+A_\ell Y_\ell'(\lambda \rho)& = 
\frac{\lambda \rho}{\ell} Y_{\ell-1}(\lambda \rho)-\frac{(\lambda \rho)^3}{2\ell^2} Y_\ell'(\lambda \rho)+O(\lambda^{-\ell+4})\nonumber \\
& = -\frac{\lambda \rho (\ell -2)!}{\ell \pi } \left( \frac{\lambda \rho}{2}\right)^{-\ell +1}-\frac{(\lambda \rho)^3 \ell !}{4\pi \ell^2 }\left(\frac{\lambda \rho}{2}\right)^{-\ell -1}+O(\lambda^{-\ell + 4}) +\delta_{\ell, 2}O(\lambda^2\log \lambda)\nonumber 
\\ & = -2 \frac{(\ell-2)!}{\pi}\left( \frac{\lambda \rho}{2}\right)^{-\ell +2}+ O( \lambda^{-\ell+4})+\delta_{\ell, 2}O(\lambda^2\log \lambda)
\end{align}
Next we observe that $J_\ell(\lambda \rho)+A_\ell J_{\ell}'(\lambda \rho)=O(\lambda^\ell)$.  
Using that $1-\frac{\ell^2}{(\lambda \rho)^2}A_\ell^2= (\lambda \rho)^2/\ell +O(\lambda^4)$ here, inserting these in \eqref{eq:siglprime} yields
\begin{align*}
\sigma_\ell'(\lambda)& = -\frac{2}{\lambda}\frac{(\lambda \rho)^{2}/\ell +O(\lambda^4)}{4[(\ell -2)!]^2 + O( \lambda^2)+ \delta_{\ell, 2}O(\lambda^2\log \lambda)}
\left( \frac{\lambda \rho}{2}\right)^{2\ell -4}\\ 
&= -\frac{\rho}{\ell [(\ell-2)!]^2}\left(\frac{\lambda \rho}{2}\right)^{2\ell -3}+O(\lambda^{2\ell -1})+  \delta_{\ell, 2}O(\lambda^3\log \lambda)
\end{align*}

For the remaining cases $J_{\ell -1}(\rho a)\not =0$.  If, in addition, $J_{\ell+1}(\rho a)\not =0$, we get the result in the third case immediately from Lemma 
\ref{l:intlgeq1}.   The second case, with $J_{\ell+1}(\rho a)=0$, follows from 
combining Lemma \ref{l:intlgeq1} with \eqref{eq:Jellp1}.
\end{proof}

The next theorem follows by combining the results of Lemmas \ref{l:s'0}, \ref{l:sigma1'}, and \ref{l:sigmaell'}, recalling Lemma \ref{l:spfundamentals}
 justifies convergence of the series.
\begin{thm}\label{t:mssp} When $\lambda \downarrow 0$, $\sigma'(\lambda)$ has the following asymptotics:
\begin{equation*}
    \sigma'(\lambda) = \begin{cases}
        \frac{-2/\lambda}{4(\log(\lambda \rho/2)  + \gamma)^2+\pi^2} + \frac{-4/\lambda}{4(\log(\lambda\rho/2) -1/2 + \gamma)^2+\pi^2} + O(\lambda/(\log\lambda)^2),&J_0(a\rho)=0\\
        -\frac{3}{2}\rho^2\lambda +O(\lambda^3\log\lambda),&J_1(a\rho)=0\\
        \frac{-2/\lambda}{4(\log(\lambda/2) + C(\rho,a) + \gamma)^2+\pi^2} +\frac{J_2(a\rho)}{J_0(a\rho)}\rho^2\lambda+ O(\lambda/(\log\lambda)^2),&\mathrm{otherwise}.
    \end{cases}
\end{equation*}
Here $C(\rho,a) = \log\rho + \frac{1}{a\rho}\frac{J_0(a\rho)}{J_1(a\rho)}$.
\end{thm}
This theorem, when combined with Lemma \ref{l:resinlimit}, proves Theorem \ref{t:spintro} and gives an explicit expression for the constants $b$ which appear
in the statement of that theorem.

\section{Numerical Methods}

The points in the plots for $\ell=1,2,3$ in Figure \ref{f:3foldres} and \ref{f:moreres} are the values $\{\operatorname{sgn}(k)k^2\Delta\varepsilon : k=-N,\ldots,N\}$, where $\Delta\varepsilon=0.0036$ and $N=15$. 
 The  $\varepsilon$ values for the $\ell=0$ plot in Figure \ref{f:ell=0} are $\{-k\Delta\varepsilon : k=2,\ldots,N\}$, with $N = 25$ and $\Delta\varepsilon = 0.1$.
 The the points on the blue curve of each plot are the approximate formulas for $\lambda_\varepsilon$, and the points on the red curve are the ``exact" values for $\lambda_\varepsilon$ obtained by applying Newton's method 20 times to locate the solution to the equation $B^{(2)}_\ell(\lambda_\varepsilon,\sqrt{a^2-\varepsilon},\rho) = 0$ with an initial guess given by the approximation from the blue curve.

\subsection*{Acknowledgments} The authors are grateful to Maciej Zworksi for suggesting studying the low energy asymptotics of the scattering phase, and for pointing out the connection to the results of \cite{haze}.
TC and KD gratefully
acknowledge partial support from Simons Collaboration Grants for Mathematicians.  KD was, in addition, 
 partially supported by NSF grant DMS-1708511. CG is supported by NSF GRFP grant DGE-2236662.

\end{document}